\documentclass{amsart} 

\usepackage[english]{babel}
\usepackage[latin1]{inputenc} 
\usepackage{amsmath}
\usepackage{amsthm}
\usepackage{amssymb}
\usepackage{tikz}
\usepackage{adjustbox}

\usepackage{imakeidx}
\usepackage{appendix}
\usepackage{tikz-cd}
\usepackage{verbatim}
\usepackage{enumitem}
\usepackage{multicol}

\usepackage[backref=page]{hyperref}

\usepackage{cleveref}

\usepackage{mathtools}

 \hypersetup{colorlinks=true,linkcolor=blue,anchorcolor=blue,citecolor=blue}

\numberwithin{equation}{section} 
\newtheorem{thm}[equation]{Theorem}

\newtheorem{prop}[equation]{Proposition}
\newtheorem{lemma}[equation]{Lemma}

\newtheorem{property}[equation]{Property}

\theoremstyle{definition}

\theoremstyle{remark}
\newtheorem{rmk}[equation]{Remark}

\newcommand{\F}{\mathbb F}
\newcommand{\Z}{\mathbb Z}
\newcommand{\Spec}{\operatorname{Spec}}
\newcommand{\G}{\mathbb G}
\renewcommand{\P}{\mathbb P}
\newcommand{\C}{\mathbb C}
\renewcommand{\c}{\subseteq}
\newcommand{\A}{\mathbb A}

\newcommand{\mc}[1]{\mathcal{#1}}
\newcommand{\cl}{\overline}
\newcommand{\set}[1]{\{#1\}}
\newcommand{\on}[1]{\operatorname{#1}}

\author{Federico Scavia}
\title[Steenrod operations on the de Rham cohomology of stacks]{Steenrod operations on the de Rham cohomology of algebraic stacks}

\address{Department of Mathematics\\
	University of British Columbia\\
	Vancouver, BC V6T 1Z2\\Canada}

\email{scavia@math.ubc.ca}

\begin{document}
	\begin{abstract}
	Building up on work of Epstein, May and Drury, we define and investigate the mod $p$ Steenrod operations on the de Rham cohomology of smooth algebraic stacks over a field of characteristic $p>0$. We then compute the action of the operations on the de Rham cohomology of classifying stacks for finite groups, connected reductive groups for which $p$ is not a torsion prime, and (special) orthogonal groups when $p=2$.
	\end{abstract}
	
	\maketitle
	
	\section{Introduction}
	Let $R$ be a commutative ring with identity, and let $X$ be a smooth algebraic stack over $R$. We will mostly be interested in the case when $R=k$ is a field and $X=BG$ is the classifying stack of a linear algebraic $k$-group $G$. (Note that $BG$ is always smooth over $k$, even if $G$ is not.) We denote by $\Omega_{X/R}$ the de Rham complex  of abelian sheaves on the big \'etale site of $X$:
		\[0\to \mc{O}_X\to \Omega^1_{X/R}\to \Omega^2_{X/R}\to\cdots.\] 	
	By definition, the de Rham cohomology $H^*_{\on{dR}}(X/R)$ of $X$ is the hypercohomology of $\Omega_{X/R}$; see \cite[\S 1]{totaro2018hodge}.

	Let  $p$ be a prime number, let $R=\F_p$ be a field of $p$ elements, and let $G$ be a finite discrete group. 
	In \cite[Lemma 10.2]{totaro2018hodge}, B. Totaro produced a canonical isomorphism of graded rings \begin{equation}\label{dr-sing}
	H^*_{\on{dR}}(BG/\F_p)\xrightarrow{\sim} H^*_{\on{sing}}(BG;\F_p)=H^*(G,\F_p),\end{equation}
	where the ring on the right is group cohomology. (The identification of the singular cohomology of $BG$ with the group cohomology of $G$ is classical.)
	 
	Let $G$ be a split connected reductive group over $\Z$. Assume that $p$ is not a torsion prime for $G$, that is, the $p$-torsion subgroup of $H^*_{\on{sing}}(BG(\C);\Z)$ is trivial, where we regard $G(\C)$ as a complex Lie group. Totaro showed in \cite[Theorem 9.2]{totaro2018hodge} that $H^*_{\on{dR}}(BG_{\F_p}/\F_p)$ is a polynomial ring on generators of degrees equal to two times the fundamental degrees of $G$. As a consequence, he obtained an isomorphism
\begin{equation}\label{dr-sing-conn}
H^*_{\on{dR}}(BG_{\F_p}/\F_p)\cong H^*_{\on{sing}}(BG(\C);\F_p).\end{equation}

	When $p$ is a torsion prime for $G$, it is an interesting problem to compute the ring $H^*_{\on{dR}}(BG_{\F_p}/\F_p)$, and to see whether (\ref{dr-sing-conn}) is still valid for $G$. For example, when $p=2$, Totaro showed in \cite[Theorem 11.1]{totaro2018hodge} that
	\begin{align}\label{dr-sing-orth}
	\nonumber H^*_{\on{dR}}(B\on{O}_{2r}/\F_2)&=\F_2[u_1,u_2,\dots,u_{2r}],\\ H^*_{\on{dR}}(B\on{O}_{2r+1}/\F_2)&=\F_2[v_1,c_1,u_2,u_3,\dots,u_{2r}]/(v_1^2),\\
	\nonumber H^*_{\on{dR}}(B\on{SO}_n/\F_2)&=\F_2[u_2,u_3,\dots,u_n],
	\end{align}
	where $|u_i|=i$ for every $i$, $|v_1|=1$ and $|c_1|=2$. In particular, when $p=2$, (\ref{dr-sing-conn}) holds for $G=\on{O}_{2r},\on{SO}_n$, but not for $G=\on{O}_{2r+1}$. (In characteristic $2$, $\on{O}_{2r}$ is disconnected, while $\on{O}_{2r+1}=\on{SO}_{2r+1}\times \mu_2$ is connected.) Moreover, he proved in \cite[Theorem 12.1]{totaro2018hodge} that (\ref{dr-sing-conn}) fails for $p=2$ and $G=\on{Spin}_{11}$. 
	
	In \cite{primozic2019computations}, E. Primozic computed $H^*_{\on{dR}}(BG_{\F_2}/\F_2)$ when $G$ is the split group of type $\on{G}_2$, and when $G=\on{Spin}_n$ for $n\leq 11$. In the examples considered by him, with the exception of $\on{Spin}_{11}$, (\ref{dr-sing-conn}) always holds. Primozic then asked whether Steenrod operations on the de Rham cohomology of smooth stacks over $\F_p$ may be defined, and whether they agree with the topological Steenrod operations on $H^*_{\on{sing}}(BG(\C);\F_p)$, when $X=BG$ and (\ref{dr-sing-conn}) holds.
	
	In the present work, we adapt a construction of R. Drury \cite{drury2019steenrod}  to define the Steenrod $p$-power operations on $H^*_{\on{dR}}(X/k)$, for a smooth algebraic stack $X$ over a field $k$ of characteristic $p$. Drury's work fits in the setting of J. P. May \cite{may1970general}. Presumably, one could also proceed by extending the work of D. Epstein \cite{epstein1966steenrod} to hypercohomology functors.
	
	We summarize the properties that we have been able to establish in \Cref{main-steenrod} below. When $k$ is perfect, we write $W_2(k)$ for the ring of Witt vectors of length $2$ with coefficients in $k$.
	
	\begin{thm}\label{main-steenrod}
	Let $p$ be a prime number, let $k$ be a field of characteristic $p$, and let $X$ be a smooth algebraic stack over $k$. Then, for all $i\in \Z$, we have group homomorphisms 
	\[\on{Sq}^i: H^*_{\on{dR}}(X/k)\to H^{*+i}_{\on{dR}}(X/k)\] when $p=2$, and \[\on{P}^i:H^*_{\on{dR}}(X/k)\to H^{*+2(p-1)i}_{\on{dR}}(X/k),\qquad \beta \on{P}^i:H^*_{\on{dR}}(X/k)\to H^{*+2(p-1)i+1}_{\on{dR}}(X/k)\] when $p>2$, respectively. These homomorphisms are natural in $X$ (i.e. they commute with pullbacks along $1$-morphisms of $k$-stacks), and satisfy the following properties.
	
	(i) For every $n\geq 0$ and every $x\in H^n_{\on{dR}}(X/k)$, we have
	\begin{equation*}
	\on{Sq}^i(x)=\begin{cases}
	x^2 &\text{ if $i=n$},\\
	0 &\text{ if $i>n$}, 
	\end{cases}
	\qquad 
	\on{P}^i(x)=\begin{cases}
	x^p &\text{ if $2i=n$},\\
	0 &\text{ if $2i>n$}, 
	\end{cases}	
	\end{equation*}
	when $p=2$ and $p>2$, respectively.
	
	(ii) The (internal) Cartan formulas of \cite[(1) p. 165]{may1970general} hold.
	
	(iii) The \'Adem relations of \cite[Theorem 4.7]{may1970general} hold.
	
	(iv) Assume that $X=[Y/G]$, where $Y$ is a smooth quasi-projective $k$-scheme and $G$ is a linear algebraic $k$-group.  Then, for all $i<0$, we have \[\on{Sq}^i=0\quad (p=2),\qquad\qquad \on{P}^i=0\quad (p>2)\] in $H^*_{\on{dR}}(X/k)$. Moreover, $\on{P}^0$ and $\on{Sq}^0$ factor as
	\[H^*_{\on{dR}}(X/k)\to H^*(X,\mc{O}_X)\to H^*(X,\mc{O}_X)\to H^*_{\on{dR}}(X/k),\] where the first map is an edge homomorphism in the Hodge spectral sequence, the second map is induced by the Frobenius endomorphism of $\mc{O}_X$ and the third map is an edge homomorphism in the conjugate spectral sequence.
	
	(v) Assume that $p>2$, that $k$ is perfect, and that $X=\tilde{X}\times_{W_2(k)}k$, where $\tilde{X}$ is a smooth $W_2(k)$-stack of finite type and with affine diagonal. Then there is a group homomorphism \[\beta: H^*_{\on{dR}}(X/k)\to H^{*+1}_{\on{dR}}(X/k)\] which is a derivation in the graded sense, and such that $\beta \on{P}^i=\beta\circ \on{P}^i$ for all $i\in \Z$.
	\end{thm}
		Our construction of the Bockstein of (v) also works for $p=2$, and in fact when $p=2$ we obtain the equality $\beta=\on{Sq}^1$ familiar from topology; see \Cref{bockstein}. Note that a Bockstein homomorphism does not exist in Epstein's generality, and it is not automatically defined in May's setting. Moreover, our Bockstein does not coincide with the one defined by Totaro during the proof of \cite[Theorem 11.1]{totaro2018hodge}; see \Cref{bockstein-totaro}.  Our construction appears to be new even in the case of smooth projective $k$-varieties. (If $X$ is a separated scheme over a ring $R$, the diagonal $X\to X\times_RX$ is a closed embedding, hence affine.)

	In (iv), there is no loss of generality in assuming that $G=\on{GL}_n$ for some $n\geq 0$.

	In \cite{may1970general}, the Cartan formula and \'Adem relations are written with homological indexing; see \cite[\S 5]{may1970general} for the changes required to pass to cohomology. Properties (i), (ii) and (iii) have been proved by Drury for the hypercohomology of a commutative differential graded $\F_p$-algebra on a topological space. Drury's proof immediately generalizes to topoi with sufficiently many points, and in particular to the big \'etale topos of $X$. The most difficult part of the proof of \Cref{main-steenrod} consists of establishing properties (iv) and (v). 
	
	Our second result is the following computation of the Steenrod operations for classifying stacks of linear algebraic groups. We denote by $\on{Sq}$ and $\on{P}$ the total Steenrod square and the total Steenrod power, respectively. By definition, for all $x\in H^*_{\on{dR}}(X/k)$, we have
	\[\on{Sq}(x)=\sum_{i\in\Z}\on{Sq}^i(x)\quad (p=2),\qquad \qquad \on{P}(x)=\sum_{i\in \Z}\on{P}^i(x)\quad (p>2).\]
	Note that by \Cref{main-steenrod}(i) the  sums contain only finitely many non-zero terms.
	
	\begin{thm}\label{mainthm}
		Let $p$ be a prime number, and let $G$ be a linear algebraic group over $\F_p$.
		\begin{enumerate}[label=(\alph*)]
			\item If $G$ is finite and discrete, then the isomorphism (\ref{dr-sing}) is compatible with Steenrod operations.
			\item If $G$ is split reductive and $p$ is not a torsion prime for $G$, then the Steenrod operations on $H^*_{\on{dR}}(BG/\F_p)$ are trivial, that is, for all $x\in H^*_{\on{dR}}(BG/\F_p)$ we have:
				\[\on{Sq}(x)=x^2\quad (p=2),\qquad \on{P}(x)=x^p\quad (p>2).\]
			\item Assume that $p=2$. The Steenrod operations on $H^*_{\on{dR}}(B\on{O}_n/\F_2)$ for $n\geq 1$, and on $H^*_{\on{dR}}(B\on{SO}_n/\F_2)$ for $n\geq 3$, are non-trivial. More precisely, we have:
			\[
			\on{Sq}(u_{2a})=u_{2a}^2,\qquad \on{Sq}(u_{2a+1})=u_{2a+1}^2+u_{4a+1}+\sum_{t=0}^{2a-1}u_{2a-t}u_{2a+1+t}.
			\]
			Here $u_i:=0$ for $i>2r$, and in the case of $\on{SO}_n$ we set $u_1:=0$. In the case of $\on{O}_{2r+1}$, we also have:
			\[\on{Sq}(v_1)=0,\qquad \on{Sq}(c_1)=c_1^2.\]
		\end{enumerate}
	\end{thm}
	In algebraic topology Steenrod operations are never trivial, because $\on{Sq}^0$ and $\on{P}^0$ are equal to the identity. Thus (b) stands in stark contrast to its topological counterpart. In general, the triviality of the Steenrod operations on de Rham cohomology is related to subtle arithmetic information; see \Cref{cris-frobenius}. 

	Since $\on{SO}_2\cong \G_{\on{m}}$, the condition $n\geq 3$ in (c) is necessary. In (c), the only interesting non-trivial operation on $u_{2a+1}$ is $\on{Sq}^{2a}$. The expression for $\on{Sq}^{2a}(u_{2a+1})$ has the same form as Wu's formula for the topological $\on{Sq}^{2a}$ applied to the Stiefel-Whitney class $w_{2a+1}$.
	
	As the computations of Totaro and Primozic show, (\ref{dr-sing-conn}) may hold even if $p$ is a torsion prime for $G$. \Cref{mainthm} shows that, even though the fact that $2$ is a torsion prime for $\on{SO}_n$ is not detected by the failure of (\ref{dr-sing-conn}), it is detected by the existence of non-trivial Steenrod operations. 
	
	We now give a brief description of the content of each section. In \Cref{definesteenrod}, we recall the definition of May's category $\mc{C}(p)$, the Steenrod operations associated to objects of $\mc{C}(p)$, and Drury's construction of the Steenrod operations on the hypercohomology of a sheaf of commutative differential graded $\F_p$-algebras $A$. In \Cref{general}, we first prove the naturality of the operations with respect of morphisms of topoi. Then, we assume that $A$ is concentrated in degree $0$, and we compare Drury's setting to that of Epstein. As a consequence, we obtain in \Cref{steenrod<0} that negative Steenrod operations on the hypercohomology of $A$ are zero, and that $\on{P}^0$ and $\on{Sq}^0$ are induced by the Frobenius endomorphism of $A$. In \Cref{cech}, we return to the case of an arbitrary commutative differential graded algebra $A$. We give a way to compute Steenrod operations using \v{C}ech cohomology; see \Cref{cech-to-derived}. In \Cref{approx}, we define Steenrod operations on de Rham cohomology of stacks, and we prove that we may approximate the de Rham cohomology of quotient stacks $X$ as in \Cref{main-steenrod}(iv) by smooth schemes, that is, for all $d\geq 0$ we construct a morphism $Z_d\to X$ such that $Z_d$ is a smooth $k$-scheme of finite type and the pullback $H^*_{\on{dR}}(X/k)\to H^*_{\on{dR}}(Z_d/k)$ is injective in degrees $\leq d$. In \Cref{derham}, we prove \Cref{mainthm}(iv) by using the results of \Cref{approx} to reduce to the case of smooth schemes, and then conclude in that case by combining the crystalline Poincar\'e Lemma with \Cref{general}. In \Cref{bockstein}, we prove \Cref{main-steenrod}(v) using the results of \Cref{cech}. Finally, in \Cref{proof} we combine the work of the previous sections with some explicit calculations to prove \Cref{mainthm}.
	
	\subsection*{Notation}
	Let $\mc{A}$ be an additive category. A cochain complex $A$ in $\mc{A}$ is an increasing sequence of objects and homomorphisms \[\dots\to A^{i-1}\xrightarrow{d^{i-1}} A^i\xrightarrow{d^i} A^{i+1}\to\cdots\] in $\mc{A}$, such that $d^i\circ d^{i-1}=0$ for all integers $i$. If $\mc{A}$ is abelian, we denote by $H^*(A)$ the cohomology of $A$: $H^i(A):=\on{Ker}d^i/\on{Im}d^{i-1}$. If $A,B\in \mc{A}$, we denote by $\on{Hom}(A,B)$ the group of homomorphisms from $A$ to $B$ in $\mc{A}$. 
	
	If $A$ is a simplicial or cosimplicial object of $\mc{A}$, we denote by $K(A)$ and $K_N(A)$ the unnormalized and normalized cochain complexes associated to $A$, respectively; see \cite[Tags 0194, 019D, 019H]{stacks-project}. When $A$ is a simplicial object, $K(A)$ and $K_N(A)$ are non-positively graded, and when $A$ is a cosimplicial object, they are non-negatively graded. 
	
	A double complex $A$ in $\mc{A}$ is a cochain complex in the category of cochain complexes of $\mc{A}$. (In particular, the squares in a double complex are commutative, not anti-commutative.) If $A$ is a double complex, we denote by $\on{Tot}(A)$ the associated total complex: \[\on{Tot}(A)^n:=\bigoplus_{i\in \Z}A^{i,n-i},\qquad d^n(a):=d^v(a)+(-1)^{n-i}d^h(a),\] where $|a|=(i,n-i)$, $d^v:A^{i,n-i}\to A^{i,n-i+1}$ is the vertical differential, and $d^h:A^{i,n-i}\to A^{i+1,n-i}$ is the horizontal differential.
	
	Let $R$ be a commutative ring with identity, let $G$ be a finite group, and let $\mc{A}$ be the category of $R$-modules with left $G$-action, i.e. the category of $R[G]$-modules. If $M$ and $N$ are $R[G]$-modules, we view the group of endomorphisms $\on{Hom}_{R}(M,N)$ and $M\otimes_{R}N$ as $R[G]$-modules in the usual way, and we let $\on{Hom}_{R[G]}(M,N)$ and $M\otimes_{R[G]}N$ be the module of $G$-equivariant homomorphisms and the tensor product of $M$ and $N$ as $R[G]$-modules. If $A$ and $B$ are cochain complexes of $R[G]$-modules, we define the cochain complex of $R[G]$-modules $\underline{\on{Hom}}_{R}(A,B)$ by
	\[\underline{\on{Hom}}_{R}^n(A,B):=\prod_{i\in\Z} \on{Hom}_{R}(A^i,B^{n+i}),\quad d^{n}(f):= \prod_{i\in \Z}(d_B^{i+n}\circ f^i+(-1)^{n+1}f^{i+1}\circ d^i_A).\]
	We define the complex of $R$-modules $\underline{\on{Hom}}_{R[G]}(A,B)$ in a similar way. Furthermore, we define their tensor product $A\otimes_{R} B$ by
	\[(A\otimes_{R} B)^n:=\bigoplus_{i\in \Z}A^i\otimes_{R} B^{n-i},\qquad d^n(f):=\bigoplus_{i\in \Z}(d^i_A\otimes \on{id}^{n-i}_B+(-1)^n(\on{id}^i_A\otimes d^{n-i}_B)).\]
	We define $A\otimes_{R[G]}B$ in a similar way. The signs are chosen so that, for all cochain complexes $A,B,C$ of $R[G]$-modules, the canonical adjunction 
	\[\Phi:\on{Hom}_{R[G]}(A,\on{Hom}_{R}(B,C))\xrightarrow{\sim} \on{Hom}_{R[G]}(A\otimes_{R}B,C)\]
	is given by the sign-free formula \[\Phi^n(f^{k+l})(a\otimes b)=(f^k(a))^l(b),\qquad |f|=n, |a|=k, |b|=l.\] We will sometimes combine the adjunction formula with the switch isomorphism \[\tau:A\otimes_RB\xrightarrow{\sim} B\otimes_RA,\qquad \tau(a\otimes b)=(-1)^{kl}(b\otimes a),\qquad |a|=k, |b|=l.\]
	
	Let $\mc{T}$ be a topos, and let $\mc{S}et$ be the set topos (also called point topos). We have a morphism of topoi $(e_*,e^{-1}):\mc{T}\to \mc{S}et$, where $e_*=\Gamma(\mc{T},-):=\on{Hom}(e,-)$, where $e$ is a terminal object of $\mc{T}$, and $e^{-1}$ is the constant sheaf functor. When writing a morphism of topoi, we do not explicitly mention the adjuction between pullback and pushforward.
	
	Let $M$ and $N$ be left $R[G]$-modules on $\mc{T}$ (here we view $\mc{T}$ as a topos ringed by the constant ring object associated to $R[G]$). We denote by $\mc{H}om_{R}(M,N)$ and $M\otimes_{R}N$ the sheaf of homomorphisms and the sheaf tensor product, respectively; see \cite[IV, Proposition 12.1, Proposition 12.7]{sga4I} for the definition. It immediately follows from the definition that $\Gamma(\mc{T},\mc{H}om_{R}(M,N))=\on{Hom}_{R}(M,N)$. 
	If $A$ and $B$ are complexes of $R$-modules on $\mc{T}$, we define the $R[G]$-modules $\underline{\mc{H}om}_{R}(A,B)$ and $A\otimes_{R}B$ on $\mc{T}$ as in the case of $R[G]$-modules (on $\mc{S}et$). It easily follows from the definition that $\Gamma(\mc{T},\underline{\mc{H}om}_{R}(A,B))=\underline{\on{Hom}}_{R}(A,B)$. We will be mostly interested in the case when $A$ is constant, in which case our definition agrees with \cite[Definition 3.2.1]{drury2019steenrod}.

By definition, a homotopy associative differential graded $R$-algebra is a cochain complex $(C,d)$ of $R$-modules, where $C^{i}=0$ for $i<0$, together with a homomorphism of complexes $m:C\otimes C\to C$ such that \[d(a\cdot b)=d(a)\cdot b+(-1)^{|a|}a\cdot  d(b),\] for all homogeneous $a$ and $b$, where $x\cdot y:=m(x\otimes y)$, and the diagram
\[
\begin{tikzcd}
C\otimes C\otimes C \arrow[d,"m\otimes \on{id}"] \arrow[r,"\on{id}\otimes m"] & C\otimes C \arrow[d,"m"] \\
C\otimes C \arrow[r,"m"] & C
\end{tikzcd}
\]
is homotopy commutative. A (two-sided) homotopy identity for $C$ is an element $u\in C^0$ such that the homomorphisms $C\to C$ of left and right multiplication by $u$ are homotopic to the identity.

\section{Steenrod operations on sheaf cohomology}\label{definesteenrod}
	We start by recalling the part of the setup of May \cite{may1970general} which is relevant to us. In \cite[\S 2]{may1970general}, the definitions are given with applications to homology in mind, and in \cite[\S 5]{may1970general} it is explained how to change the indexing when dealing with cohomology. We give the definitions directly in the cohomological setting.
	
	Let $R$ be a commutative ring with identity, let $p$ be a prime number, let $\Sigma_p$ be the symmetric group on $p$ letters, and let $\pi$ be the cyclic subgroup of $\Sigma_p$ generated by some $\alpha\in \Sigma_p$ of order $p$. Let $\epsilon:W(p,R)\to R$ be the $R[\pi]$-free resolution of $R$ defined as follows. For every $i\geq 0$, we let $W(p,R)_i=W(p,R)^{-i}$ be $R[\pi]$-free on one generator $e_i$. Define $N,T\in R[\pi]$ by the formulas \[N:=1+\alpha+\dots+\alpha^{p-1},\qquad T:=\alpha-1.\] We define a differential $d$ and the augmentation by the formulas
	\[d(e_{2i+1}):=Te_{2i},\qquad d(e_{2i+2}):=Ne_{2i+1},\qquad \epsilon(\alpha^je_0):=1\]
	for all $i,j\geq 0$.
	
	We also let $V$ be a $R[\Sigma_p]$-free resolution of $R$ such that $V^{i}=0$ for all $i>0$. In particular, $V$ is a $R[\pi]$-free resolution of $R$, hence by \cite[Tag 0649]{stacks-project} there exists a homomorphism $j:W(p,R)\to V$ of $R[\pi]$-complexes, which commutes with the augmentations maps.

	\subsection{The category $\mc{C}(p)$.}\label{cp} Following \cite[Definitions 2.1]{may1970general}, we define a category $\mc{C}(p,R)$ as follows. The objects of $\mc{C}(p,R)$ are pairs $(K,\theta)$, where $K$ is a homotopy associative differential graded $R$-algebra, $\pi$ acts trivially on $K$ and on $K^{\otimes p}$ by cyclic permutations of the factors, and $\theta: W(p,R)\otimes K^{\otimes p}\to K$ is a homomorphism of $R[\pi]$-complexes such that:
		\begin{enumerate}[label=(\roman*)]
			\item the restriction of $\theta$ to $(R e_0)\otimes K^{\otimes p}\cong K^{\otimes p}$ is $R$-homotopic to the iterated product $K^{\otimes p}\to K$ (associated in some order), and
			\item there exists a homomorphism $\phi: V\otimes K^{\otimes p}\to K$ of $R[\Sigma_p]$-complexes such that $\theta$ is $R[\pi]$-homotopic to the composition \[W(p,R)\otimes K^{\otimes p}\xrightarrow{j\otimes 1}V\otimes K^{\otimes p}\xrightarrow{\phi} K.\]
		\end{enumerate} 
		Note that (ii) does not depend on the choice of $V$ or $j$. A morphism $f:(K_1,\theta_1)\to (K_2,\theta_2)$ in $\mc{C}(p,R)$ is defined as a homomorphism of $R$-complexes $g:K_1\to K_2$ such that the diagram
		\[
		\begin{tikzcd}
		W(p,R)\otimes K_1^{\otimes p} \arrow[d,"\on{id}\otimes g^{\otimes p}"] \arrow[r,"\theta_1"] & K_1 \arrow[d,"g"] \\
		W(p,R)\otimes K_2^{\otimes p} \arrow[r,"\theta_2"] & K_2.
		\end{tikzcd}
		\]
		is $R[\pi]$-homotopy commutative. We write \[\mc{C}(p):=\mc{C}(p, \Z/p\Z),\qquad W:=W(p,\Z/p\Z).\]		
		Let $i:R^{\otimes p}\xrightarrow{\sim}R$ be the isomorphism $a_1\otimes\dots \otimes a_p\mapsto a_1\cdots a_p$. Note that $(R,\epsilon\otimes i)$ is an object of $\mc{C}(p,R)$. We say that $(K,\theta)\in \mc{C}(p,R)$ is \emph{unital} if $K$ has a two-sided homotopy identity $u$ such that the homomorphism $R \to K$ given by $1\mapsto u$ is a morphism in $\mc{C}(p,R)$.
		
		If $R\to S$ is a ring homomorphism, and $(K,\theta)$ is an object of $\mc{C}(p,R)$, the pair $(K\otimes_RS,\theta\otimes_R\on{id}_S)$ defines an object of $\mc{C}(p,S)$. We say that $(K,\theta)\in \mc{C}(p)$ is \emph{reduced} if there exists an object $(\tilde{K},\tilde{\theta})\in \mc{C}(p,\Z/p^2\Z)$ whose reduction modulo $p$ is isomorphic to $(K,\theta)$, and such that $\tilde{K}$ is flat over $\Z/p^2\Z$. If $(K,\theta)$ is reduced, tensorization  by $\tilde{K}$ of the short exact sequence $0\to \Z/p\Z\to \Z/p^2\Z\to \Z/p\Z\to 0$ gives rise  to a short exact sequence of cochain complexes \[0\to K\to \tilde{K}\to K\to 0,\] where $\tilde{K}\to K$ is the natural projection. The associated connecting map \[\beta:H^*(K)\to H^{*+1}(K)\] is called the Bockstein homomorphism.
	
	\subsection{The Steenrod operations.}\label{def-steenrod} Given a prime $p$ and an object $(K,\theta)$ of $\mc{C}(p)$, May constructs mod $p$ Steenrod operations on $H^*(K)$ as follows; see \cite[Definitions 2.2]{may1970general}. The homomorphism $\theta$ induces a map \[\theta_*:H^*(W\otimes_{\F_p[\pi]}K^{\otimes p})\to H^*(K).\] Let $q,i\geq 0$ be integers, and let $x\in H^q(K)$. Define \[D_i(x):=\theta_*(e_i\otimes x^{\otimes p})\in H^{pq-i}(K).\] It is checked in \cite[p. 161]{may1970general} (using homological indexing) that this construction does not depend on the choice of representative of $x$. If $i<0$, we set $D_i$ equal to zero. When $p>2$, for every integer $n$ we let \[\nu(n):=(-1)^{n(n-1)(p-1)/4}((p-1)/2)!.\] The mod $p$ Steenrod operations on $H^*(K)$ are defined by the following formulas.
	\begin{itemize}
		\item[--] If $p=2$, $\on{Sq}^s(x):=D_{q-s}(x)\in H^{q+s}(K)$;
		\item[--] If $p>2$, $\on{P}^s(x):=(-1)^s\nu(-q)D_{(q-2s)(p-1)}(x)\in H^{q+2s(p-1)}(K)$, and $\beta \on{P}^s(x):=(-1)^s\nu(-q)D_{(q-2s)(p-1)-1}(x)\in H^{q+2s(p-1)+1}(K)$. 
	\end{itemize}
	One then extends the definitions on arbitrary elements of $H^*(K)$ by linearity. Note that this defines Steenrod operations $\on{Sq}^s$ and $\on{P}^s$ for every integer $s$. 
	For every $x\in H^*(K)$, we define \[\on{Sq}(x):=\sum_s \on{Sq}^s(x),\qquad \on{P}(x):=\sum_s\on{P}^s(x).\]

	 There is not a construction of a Bockstein homomorphism $\beta$ on $H^*(K)$ taking as input an arbitrary $(K,\theta)$. In particular, the expression $\beta P^s$ appearing in the definition is a single symbol, and not a composition. 
	 
	\begin{lemma}\label{bockstein-properties}
	 	Let $(K,\theta)$ be a reduced object of $\mc{C}(p)$, and let $\beta:H^*(K)\to H^{*+1}(K)$ be the Bockstein homomorphism. Then $\beta$ satisfies the following properties.
	 	
	 	(a) It is a derivation in the graded sense: for every two homogeneous elements $a,b$ of $H^*(K)$, we have \[\beta(a\cdot b)=\beta(a)\cdot b+(-1)^{|a|}a\cdot \beta(b).\]
	 	
	 	(b) When $p=2$, we have $\beta=\on{Sq}^1$. When $p>2$, we have $\beta\on{P}^s=\beta\circ \on{P}^s$ for all $s\in \Z$.
	 \end{lemma}
 
 	\begin{proof}
 	(a) This follows from the definition of $\beta$ as a connecting homomorphism, together with the fact that the differential of $\tilde{K}$ is a derivation in the graded sense.
 	
 		
 	(b) This is proved in \cite[Proposition 2.5(v), Corollary 2.3]{may1970general}.
 	\end{proof}
 	
	Let $(K',\theta')$ be another object of $\mc{C}(p)$, and let $f:K\to K'$ be a homomorphism of complexes. If $f$ is a morphism in $\mc{C}(p)$, then the induced homomorphism $f_*:H^*(K)\to H^*(K')$ respects Steenrod operations.

	\begin{rmk}\label{associative}
	If $K$ is associative and $\theta=\epsilon \otimes m_p$, where $m_p:K^{\otimes p}\to K$ is the $p$-fold iterated product, then $(K,\theta)$ is an object of $\mc{C}(p)$. Moreover, it immediately follows from the definition that the $p$-th power Steenrod operations on $H^*(K)$ are trivial, that is, for all $x\in H^*(K)$ we have $\on{Sq}(x)=x^2$ (if $p=2$) or $\on{P}(x)=x^p$ (if $p>2$). 
	\end{rmk}
	
	\subsection{Operations on sheaf cohomology.}We now review the construction of the Steenrod operations on sheaf cohomology given by  Drury in \cite{drury2019steenrod}. Let $p$ be a prime number, let $\mc{T}$ be a topos with sufficiently many points,\footnote{Drury works in the context of sheaves on a topological space. As we will see, his definitions and arguments easily adapt to topoi with sufficiently many points.} and let $(A,d)$ be a commutative differential graded $\F_p$-algebra on $\mc{T}$. We denote by $\mathbb{H}^*(\mc{T},A)$ the hypercohomology of $A$, that is, $\mathbb{H}^*(\mc{T},A):= H^*(\Gamma(\mc{T},I))$, where $A\to I$ is an injective resolution of $\F_p$-vector spaces on $\mc{T}$. We denote by $m:A\otimes A\to A$ the multiplication map.
	
	 Let $\nu:A\to I$ be an injective resolution of $A$ in the category of $\F_p$-vector spaces on $\mc{T}$. By \cite[Tag 013K]{stacks-project}, we may choose the resolution to be \emph{standard}: $\nu$ is a monomorphism in every degree, and $I^i=0$ for $i<0$.
	
	Since we are working over a field, for every $r\geq 2$ the $r$-fold tensor product
	$\nu^{\otimes r}:A^{\otimes r}\to I^{\otimes r}$ is a monomorphism in every degree and a quasi-isomorphism; see \cite[Lemma 3.2.16]{drury2019steenrod}. Here we use the fact that $\mc{T}$ has sufficiently many points, because then exactness may be checked on stalks. By \cite[Tag 013P]{stacks-project}, there exists $\tilde{m}:I\otimes I\to I$ making the following diagram commute:
	
	\[
	\begin{tikzcd}
	A\otimes A \arrow[d,"\nu\otimes \nu"] \arrow[r,"m"] & A \arrow[d,"\nu"] \\
	I\otimes I \arrow[r,"\tilde{m}"] & I
	\end{tikzcd}
	\]
	By \cite[Tag 013S]{stacks-project}, $\tilde{m}$ is unique up to homotopy. 
	
	We define $K:=\Gamma(\mc{T},I)$, with  differential induced by that of $I$. We define a homomorphism $M:K\otimes K\to K$ in degree $n\geq 0$ as the composition \[\oplus_q\Gamma(\mc{T},I^{q})\otimes \Gamma(\mc{T},I^{n-q})\to \Gamma(\mc{T},\oplus_q I^q\otimes I^{n-q})\xrightarrow{\Gamma(\tilde{m})}\Gamma(\mc{T},I^n).\]
		For every $a,b\geq 0$, we have a cup product given by the composition \[\cup: H^a(K)\otimes H^b(K)\to H^{a+b}(K\otimes K)\xrightarrow{H^{a+b}(M)}  H^{a+b}(K).\]	
	
	By \cite[Lemma 5.1.3]{drury2019steenrod}, the product $M:K\otimes K\to K$ makes $K$ into a homotopy associative differential graded $\F_p$-algebra, and the induced cup product $\cup : H^*(K)\otimes H^*(K)\to H^*(K)$ is graded commutative and associative.
	
	The natural map \[A=\underline{\mc{H}om}_{\F_p}(e^{-1}\F_p,A)\to \underline{\mc{H}om}_{\F_p}(e^{-1}W,I)\] is an $\F_p[\pi]$-resolution by $\F_p[\pi]$-injective objects; see \cite[Corollary 4.3.4]{epstein1966steenrod} for the case when $A$ is concentrated in degree $0$, and \cite[Lemma 3.2.21]{drury2019steenrod} for the general case. It follows that there exists a commutative diagram of $\pi$-equivariant maps
	\begin{equation}\label{key-diagram}
	\begin{tikzcd}
	A^{\otimes p} \arrow[d,"\nu^{\otimes p}"] \arrow[r,"m_p"] & A \arrow[d] \\
	I^{\otimes p} \arrow[r,"\tilde{m}_p"] & \underline{\mc{H}om}_{\F_p}(e^{-1}W,I).
	\end{tikzcd}
	\end{equation}
	The map $\tilde{m}_p$ is unique up to a $\pi$-equivariant homotopy. Passing to global sections, we obtain an $\F_p[\pi]$-homomorphism
	\[\Gamma(\mc{T}, I)^{\otimes p}\to \Gamma(\mc{T},I^{\otimes p})\xrightarrow{\Gamma(\tilde{m}_p)} \Gamma(\mc{T}, \underline{\mc{H}om}_{\F_p}(e^{-1}W,I))= \underline{\on{Hom}}_{\F_p}(W,\Gamma(\mc{T},I)),\] where the last equality is the projection formula \cite[IV, Proposition 10.3]{sga4I}. By the tensor-hom adjunction, this is equivalent to a $\pi$-equivariant homomorphism \[\theta: W\otimes \Gamma(\mc{T}, I)^{\otimes p}\to \Gamma(\mc{T},I).\]
	By \cite[Lemmas 5.3.1 and 5.3.2]{drury2019steenrod}, the pair $(K,\theta)$ satisfies the properties (i) and (ii) of \Cref{cp}, and so is an object in May's category $\mc{C}(p)$. We thus obtain Steenrod operations on $\mathbb{H}^*(\mc{T},A)$. It is straightforward to verify that the operations do not depend on the choice of $\nu:A\to I$, $\tilde{m}$, and $\tilde{m}_p$.
	
	\begin{prop}\label{sheaf-properties}
	Let $\mc{T}$ be a topos with sufficiently many points, and let $A$ be commutative differential graded $\F_p$-algebra on $\mc{T}$.
		
	(i) For every $n\geq 0$ and every $x\in \mathbb{H}^*(\mc{T},A)$, we have
	\begin{equation*}
	\on{Sq}^i(x)=\begin{cases}
	x^2 &\text{ if $i=n$},\\
	0 &\text{ if $i>n$}, 
	\end{cases}
	\qquad 
	\on{P}^i(x)=\begin{cases}
	x^p &\text{ if $2i=n$},\\
	0 &\text{ if $2i>n$}, 
	\end{cases}	
	\end{equation*}
	when $p=2$ and $p>2$, respectively.
	
	(ii) The (internal) Cartan formulas of \cite[(1) p. 165]{may1970general} hold.
	
	(iii) The \'Adem relations of \cite[Theorem 4.7]{may1970general} hold.	
	\end{prop}	

	\begin{proof}
	Property (i) holds for the Steenrod operations on $H^*(K)$, for every object $(K,\theta)\in \mc{C}(p)$, as is easily seen by unwinding the definitions of $\on{Sq}^i$ and $\on{P}^i$.
	
	Let now $(K,\theta)\in \mc{C}(p)$ be an object associated to $A$ by the construction above. By \cite[Lemma 5.3.3, Lemma 5.3.4]{drury2019steenrod}, $(K,\theta)$ is a Cartan object and an \'Adem object. We refer the reader to \cite[Definitions 2.1, Definition 4.1]{may1970general} for the definition of Cartan and \'Adem objects. It now follows from \cite[(1) p. 165, Theorem 4.7]{may1970general} that (ii) and (iii) are also satisfied.
	\end{proof}
		
	\begin{rmk}
		Let $u\in I^0$ be the image of the multiplicative identity of $A$. Then $u$ is a homotopy identity for $\Gamma(\mc{T},I)$, and the homomorphism $\F_p\to I$ given by $1\mapsto u$ induces a morphism $(\F_p,\epsilon\otimes i)\to (\Gamma(\mc{T},I),\theta)$ in $\mc{C}(p)$. Thus $(\Gamma(\mc{T},I),\theta)$ is unital. By \cite[Proposition 3.1(iii)]{may1970general}, if we let $1\in H^*(K)$ denote the multiplicative unit, we have $\on{Sq}(1)=1$ and $\on{P}(1)=1$.
	\end{rmk}
	
	\section{General properties}\label{general}
	In this section, we establish properties of the Steenrod operations on $\mathbb{H}^*(\mc{T},A)$ which hold in wide generality.	In later sections, we will restrict our attention to the case when $\mc{T}$ is the big \'etale site of a smooth algebraic stack $X$ over a field of positive characteristic, and $A$ is the de Rham complex of $X$.
		
	\subsection{Functoriality}
	Let $f=(f_*,f^{-1}):\mc{T}'\to \mc{T}$ be a morphism of topoi. Let $(A',d')$ and  $(A,d)$ be commutative differential graded $\F_p$-algebras on $\mc{T'}$ and on $\mc{T}$, respectively. Then $f^{-1}A$ is a commutative differential graded $\F_p$-algebra over $\mc{T}$; see \cite[IV, 3.1.2]{sga4I}. Let $f^{-1}A\to A'$ be a homomorphism of commutative differential graded $\F_p$-algebras on $\mc{T}$. There is a induced homomorphism
	\begin{equation}\label{functorial-map}
	\mathbb{H}^*(\mc{T},A)\to \mathbb{H}^*(\mc{T}',A')
	\end{equation}
	which is constructed as follows. Let $A\to I$ and $A'\to I'$ be standard injective resolutions of $\F_p$-vector spaces on $\mc{T}$. Since $f^{-1}$ is exact, $f^{-1}A\to f^{-1}I$ is a quasi-isomorphism. By \cite[Tag 013P]{stacks-project}, there exists a morphism $f^{-1}I\to I'$ making the square 
	\[
	\begin{tikzcd}
	f^{-1}A \arrow[d] \arrow[r] & A' \arrow[d] \\
	f^{-1}I \arrow[r] & I'
	\end{tikzcd}
	\]
	commute. By adjunction, we get a map $I\to f_*(I')$, and taking global sections yields (\ref{functorial-map}). It is easy to check that (\ref{functorial-map}) does not depend on the choice of resolutions $I,I'$, and of the map $f^{-1}I\to I'$.

	\begin{lemma}\label{functorial}
	 The homomorphism (\ref{functorial-map}) is compatible with Steenrod operations.
	\end{lemma}

	\begin{proof}	
	 There is a canonical $\Sigma_r$-equivariant isomorphism $f^{-1}(A^{\otimes r})\xrightarrow{\sim} (f^{-1}A)^{\otimes r}$ for every $r\geq 2$; see \cite[Proposition 13.4(c)]{sga4I}. It follows that for every $r\geq 2$ we have a commutative $\Sigma_r$-equivariant diagram
	\[
	\begin{tikzcd}
	f^{-1}A^{\otimes r} \arrow[d] \arrow[r] & f^{-1}A \arrow[d] \\
	(A')^{\otimes r} \arrow[r] & A',
	\end{tikzcd}
	\]
	where the horizontal maps are the $r$-fold multiplication maps. Denote by $(e'_*,(e')^{-1}):\mc{T}'\to \mc{S}et$ the global sections morphism for $\mc{T}'$. We obtain a diagram of $\pi$-equivariant homomorphisms
	\[
	\begin{tikzcd}
	f^{-1}I^{\otimes p} \arrow[ddd] \arrow[rrr] &&& f^{-1}\underline{\mc{H}om}_{\F_p}(e^{-1}W,I) \arrow[ddd] \\
	& f^{-1}A^{\otimes p}\arrow[r] \arrow[d]  \arrow[ul]  & f^{-1}A \arrow[ur] \arrow[d]  & \\
	& (A')^{\otimes p} \arrow[dl] \arrow[r]  & A' \arrow[dr] &\\ 
	(I')^{\otimes p} \arrow[rrr] &&& \underline{\mc{H}om}_{\F_p}((e')^{-1}W,I')
	\end{tikzcd}
	\]
	where each of the five inner squares is commutative. To construct the left and right squares, apply \cite[Tag 013P]{stacks-project}. The bottom square is (\ref{key-diagram}) for $A'$, and the top square is the pullback of (\ref{key-diagram}) for $A$. Thus, since $f^{-1}$ is exact, the four diagonal arrows are $\pi$-equivariant quasi-isomorphisms. We deduce that the outer square is commutative in the derived category of $\pi$-equivariant $\F_p$-vector spaces on $\mc{T}$. Since every term of $\underline{\mc{H}om}_{\F_p}((e')^{-1}W,I')$ is injective, by \cite[Tag 05TG]{stacks-project} the outer square is $\pi$-homotopy commutative.
	Using the adjunction between $f^{-1}$ and $f_*$, we obtain a $\pi$-homotopy commutative diagram
	\[
	\begin{tikzcd}
	I^{\otimes p} \arrow[d] \arrow[r] & \underline{\mc{H}om}_{\F_p}(e^{-1}W,I) \arrow[d] \\
	f_*((I')^{\otimes p}) \arrow[r] & f_*\underline{\mc{H}om}_{\F_p}((e')^{-1}W,I').
	\end{tikzcd}
	\]
	We thus get a diagram 
	\[
	\begin{tikzcd}
	\Gamma(\mc{T},I)^{\otimes p}\arrow[r] \arrow[d] &  \Gamma(\mc{T},I^{\otimes p}) \arrow[d] \arrow[r] & \underline{\on{Hom}}_{\F_p}(W,\Gamma(\mc{T},I)) \arrow[d]\\
	\Gamma(\mc{T}',I')^{\otimes p} \arrow[r] & \Gamma(\mc{T}',(I')^{\otimes p}) \arrow[r] & \underline{\on{Hom}}_{\F_p}(W,\Gamma(\mc{T}',I')),
	\end{tikzcd}
	\]
	where the square on the left is $\pi$-equivariantly commutative, and the square on the right is obtained by taking global sections in the square above and using the projection formula \cite[IV, Proposition 10.3]{sga4I}, and so it is $\pi$-homotopy commutative. Using the tensor-hom adjunction on the outer rectangle in the previous diagram, we deduce that the square
	\[
	\begin{tikzcd}
	W\otimes\Gamma(\mc{T},I)^{\otimes p} \arrow[d] \arrow[r,"\theta"] & \Gamma(\mc{T},I) \arrow[d] \\
	W\otimes \Gamma(\mc{T}',I')^{\otimes p} \arrow[r,"\theta'"] & \Gamma(\mc{T}',I').
	\end{tikzcd}
	\]	
	is $\pi$-homotopy commutative. This means that (\ref{functorial-map}) gives a morphism in $\mc{C}(p)$, and so it is compatible with Steenrod operations.	
	\end{proof}
	
	\begin{rmk}\label{commutes-ruu}
	Consider the special case when $A'=f^{-1}A$. We have a commutative square
	\[
	\begin{tikzcd}
	\mathbb{H}^*(\mc{T},A) \arrow[r] \arrow[d] & \mathbb{H}^*(\mc{T}',f^{-1}A)   \arrow[d,"\wr"] \\
	\mathbb{H}^*(\mc{T},f_*f^{-1}A)  \arrow[r] &  \mathbb{H}^*(\mc{T},Rf_*f^{-1}A),  
	\end{tikzcd}
	\]
	where the top horizontal map is (\ref{functorial-map}), the vertical map on the left is induced by the unit $A\to f_*f^{-1}A$, and the bottom horizontal map is induced by the natural map $f_*f^{-1}A\to Rf_*f^{-1}A$.  The vertical isomorphism on the right is obtained using the fact that $u_*$ respects injectives. The proof is a simple exercise in homological algebra; we leave it to the reader. This diagram will be used during the proof of \Cref{cris-compatible-scheme}. 
	\end{rmk}
		
	\subsection{The case when $A$ is concentrated in degree $0$}\label{deg0epstein}
	Let $\mc{T}$ be a topos, and let $q=(q_*,q^{-1})$ be a point of $\mc{T}$, that is, a morphism from the set topos $\mc{S}et$ to $\mc{T}$; see \cite[IV, D\'efinition 6.1]{sga4I}. If $E$ is a set, $q_*E$ is called a skyscraper sheaf; see \cite[Tag 00Y9]{stacks-project}. 
	
	\begin{lemma}\label{actually-inj}
		Let $k$ be a field, and let $F=\prod_{j\in J}q_{j*}V_j$, where $J$ is a set and, for every $j\in J$, $V_j$ is a $k$-vector space and $q_j$ is a point of $\mc{T}$. Then $F$ is an injective object of the category of sheaves of $k$-vector spaces on $\mc{T}$. 
	\end{lemma}
	Note that, in \Cref{actually-inj}, $J$ is allowed to be infinite, and the $V_j$ may be infinite dimensional over $k$.
	\begin{proof}
		If $\set{F_j}_{j\in J}$ is a collection of injective $k$-vector spaces on $\mc{T}$, then $\prod_j F_j$ is also injectve. Therefore, it suffices to show that $q_{j*}V_j$ is injective for every $j\in J$. In the category of $k$-vector spaces, every object is injective. Note that $q_{j*}$ is a right adjoint even when viewed as a functor between the associated ringed topoi of $k$-vector spaces; see \cite[IV, Proposition 13.4]{sga4I}. It follows that each $q_{j*}$ preserves injective objects, and the conclusion follows.
	\end{proof}
	
	Assume now that $\mc{T}$ has sufficiently many points, and let $\set{q_i}_{i\in I}$ be a conservative family of points of $\mc{T}$; see \cite[Tag 00YK]{stacks-project}. Let $k$ be a field. There is an endofunctor of the category of $k$-vector spaces on $\mc{T}$ given on objects $V$ by \[S(V):=\prod_{i\in I}q_{i*}q_i^{-1}V,\] and defined in an obvious way on morphisms. Iterating $S$, for every $k$-vector space $V$ on $\mc{T}$ we obtain a cosimplicial $k$-vector space $S^*(V)$; see \cite[8.1.4]{epstein1966steenrod}. Applying the associated (non-normalized) cochain complex functor $K(-)$, we get a quasi-isomorphism $V\to K(S^*V)$. By \Cref{actually-inj}, every term of $K(S^*V)$ is injective, and so $V\to K(S^*V)$ is an injective resolution of $k$-vector spaces on $\mc{T}$. It may be helpful for the reader to note that this is a special case of \cite[Theorem 8.1.5]{epstein1966steenrod}.
	
	Let now $k=\F_p$, and let $A$ be an $\F_p$-algebra on $\mc{T}$. In \cite[(8.3.2)]{epstein1966steenrod}, a $\pi$-equivariant map 
	\begin{equation}\label{epstein-map}
	K(S^*A)^{\otimes p}\to \underline{\mc{H}om}(e^{-1}W,K(S^*A))
	\end{equation}
	is constructed. We will need the details of the construction of (\ref{epstein-map}) during the proof of \Cref{epstein=may} below. We have a sequence of $\pi$-equivariant cochain maps
	\begin{align}\label{epstein-equiv}W(p,\Z)\to \underline{\on{Hom}}(\nabla k,t)\to \underline{\on{Hom}}(K&(S^*A)^{\otimes p}, K((S^*A)^{\otimes p})\to \\ &\to \underline{\on{Hom}}(K(S^*A)^{\otimes p}), K(S^*A)),\nonumber\end{align} defined as follows.
	\begin{itemize}
		\item[--] The first map (from left to right) is the composition of two homomorphisms $\Psi:W(p,\Z)\to \underline{\on{Hom}}(\nabla k,t)$ and $\Theta:\underline{\on{Hom}}(\nabla k,t)\to \underline{\on{Hom}}(\nabla k,t)$, both defined in \cite[p. 217]{epstein1966semisimplicial}.\footnote{In \cite[p. 217]{epstein1966steenrod}, $W(p,\Z)$ is denoted by $\mathcal{W}$. The complex $\underline{\on{Hom}}(\nabla k,t)$ is defined in \cite[3.1]{epstein1966semisimplicial}, where it is denoted by $\mc{H}om(\nabla k,t)$.}
		\item[--] The second map is induced from the composition $E_A\hat{M}$ of the functors $\hat{M}$ and $E_A$, as defined in \cite[p. 193]{epstein1966steenrod}. That this is a homomorphism of complexes is checked in \cite[p. 194]{epstein1966steenrod}.
		\item[--]  The multiplication map $A^{\otimes p}\to A$ induces a cosimplicial map $S^*(A^{\otimes p})\to S^*A$, hence a cochain map $K((S^*A)^{\otimes p})\to K(S^*A)$. Consider the cosimplicial map $(S^*A)^{\otimes p}\to S^*(A^{\otimes p})$ given by \cite[Lemma 8.2.3]{epstein1966steenrod}, and the induced cochain map $K(S^*A)^{\otimes p}\to K(S^*(A^{\otimes p}))$. The third map is induced by the composition $K(S^*A)^{\otimes p}\to K(S^*(A^{\otimes p}))\to K(S^*A)$.
	\end{itemize}
	Since $A$ is $p$-torsion, (\ref{epstein-equiv}) factors through the projection $W(p,\Z)\to W$. Using the adjunction between $e_*$ and $e^{-1}$ and the tensor-hom adjunction, we get the following canonical isomorphisms:
	\begin{align}\label{adjunction}\nonumber\on{Hom}(W,\underline{\on{Hom}}(K(S^*A)^{\otimes p}, K(S^*A)))&=\on{Hom}(e^{-1}W,\underline{\mc{H}om}(K(S^*A)^{\otimes p},K(S^*A)))\\ &=\on{Hom}(e^{-1}W\otimes K(S^*A)^{\otimes p},K(S^*A))\\ \nonumber &=\on{Hom}(K(S^*A)^{\otimes p},\underline{\mc{H}om}(e^{-1}W,K(S^*A))).\end{align}
	We define (\ref{epstein-map}) to be the image of (\ref{epstein-equiv}) under this chain of isomorphisms.
	
	On the other hand, applying (\ref{key-diagram}) with $I=K(S^*A)$, we have another homomorphism
	\begin{equation}\label{drury}K(S^*A)^{\otimes p} \to \underline{\mc{H}om}(e^{-1}W,K(S^*A)).\end{equation}  
	\begin{lemma}\label{epstein=may}
		The homomorphisms (\ref{epstein-map}) and (\ref{drury}) are homotopic. 	
	\end{lemma}
	
	\begin{proof}
		We construct the following commutative diagram of $\pi$-equivariant maps
		\begin{center}
					\adjustbox{scale=0.77}{
				\begin{tikzcd}
					W(p,\Z)_0 \arrow[dd,"\epsilon"] \arrow[r] & \on{Hom}(\nabla k,t) \arrow[dd] \arrow[r] & \on{Hom}(K(S^*A)^{\otimes p}, K((S^*A)^{\otimes p})) \arrow[r] \arrow[d] & \on{Hom}(K(S^*A)^{\otimes p}, K(S^*A)) \arrow[d]  \\
					&& \on{Hom}(A^{\otimes p}, A^{\otimes p}) \arrow[r] \arrow[d]  & \on{Hom}(A^{\otimes p}, A) \arrow[d] \\
					\Z \arrow[r,"\sim"] & \on{Hom}((\nabla k)_0,t_0) \arrow[r] \arrow[ur, dashed] & \on{Hom}(S(A)^{\otimes p}, S(A)^{\otimes p}) \arrow[r] & \on{Hom}(S(A)^{\otimes p}, S(A)).  
				\end{tikzcd}
			}
		\end{center}

		Here, the top row is obtained by applying the functor $Z^0(-)$ of $0$-cocycles to (\ref{epstein-equiv}). 
		
		If $C_1$ and $C_2$ are cochain complexes in an additive category $\mc{A}$, there is a natural map \[\on{Hom}(C_1,C_2)\to \on{Hom}((C_1)^0,(C_2)^0).\] Moreover, if $\mc{A}$ is abelian, this map factors as
		\[\on{Hom}(C_1,C_2)\to \on{Hom}(Z^0(C_1),Z^0(C_2))\to \on{Hom}((C_1)^0,(C_2)^0).\]
		All these homomorphisms are contravariantly functorial in $C_1$ and covariantly functorial in $C_2$. This is how the square on the left and the two squares on the right are constructed. The penthagon of solid arrows in the center is induced by the functor $E_A\hat{M}$. This completes the construction of the solid arrows in the previous diagram.
		
		We have $(\nabla k)_0=t_0=M(0,\dots,0)$ (see \cite[End of p. 213]{epstein1966semisimplicial} for the notation), and the bottom map $\Z\to \on{Hom}((\nabla k)_0,t_0)$ sends $1$ to the identity. Hence the composition $\Z\to \on{Hom}(S(A)^{\otimes p},S(A)^{\otimes p})$ sends $1$ to the identity. Since $\on{id}\in \on{Hom}(A^{\otimes p},A^{\otimes p})$ is $\pi$-invariant, we obtain the dashed arrow. 
		
		The composition $\eta:\Z\to \on{Hom}(A^{\otimes p}, A)$ of the bottom row sends $1$ to the multiplication map $A^{\otimes p}\to A$. Since $A$ is $p$-torsion, the composition of the top row factors through $W_0$, and $\Z\to \on{Hom}(A^{\otimes p},A)$ factors through $\F_p$. This proves that the square of $\pi$-equivariant maps
		\[
		\begin{tikzcd}
		{W}_0 \arrow[r] \arrow[d,"\epsilon"] & \on{Hom}(K(S^*A)^{\otimes p}, K(S^*A)) \arrow[d]  \\
		\F_p \arrow[r,"\eta"] & \on{Hom}(A^{\otimes p},A) 
		\end{tikzcd}
		\]
		is commutative. 
		We have adjunction isomorphisms
		\begin{align*}
		\on{Hom}(\F_p,\underline{\on{Hom}}(A^{\otimes p},A))&= \on{Hom}(e^{-1}\F_p,\underline{\mc{H}}om(A^{\otimes p},A))\\
		&=\on{Hom}(e^{-1}\F_p\otimes A^{\otimes p},A)\\
		&=\on{Hom}(A^{\otimes p},\underline{\mc{H}om}(e^{-1}\F_p,A)),		
		\end{align*}
		which are compatible with those of (\ref{adjunction}). Applying the adjunctions to the above square, we get a commutative square of $\pi$-equivariant homomorphisms
		\[
		\begin{tikzcd}
		K(S^*A)^{\otimes p} \arrow[r] & \underline{\mc{H}om}(e^{-1}W,K(S^*A)) \\
		A^{\otimes p} \arrow[r,"m_p"] \arrow[u]  & A, \arrow[u]
		\end{tikzcd}
		\]
		where the top horizontal map is (\ref{epstein-map}), the left vertical map is a monomorphism in every degree and a quasi-isomorphism, and the right vertical map is a $\pi$-injective resolution. By definition, the map (\ref{drury}) fits in a square (\ref{key-diagram}) of the same form. The conclusion now follows from \cite[Tag 013S]{stacks-project}. 		
	\end{proof}
	
	\begin{prop}\label{steenrod<0}
		Let $A$ be an $\F_p$-algebra on $\mc{T}$. 
	\begin{enumerate}[label=(\alph*)]
		\item All negative Steenrod operations on $\mathbb{H}^*(\mc{T},A)$ are zero.
		\item The operations $\on{Sq}^0$ and $\on{P}^0$ are induced by the Frobenius endomorphism $A\to A$.
	\end{enumerate}		
	\end{prop}
	
	\begin{proof}
		(a) It suffices to show that $D_i(x)=0$ for every $x\in H^q(K(S^*A))$ and every $i>q(p-1)$ ($i\geq (q+2)(p-1)$ would be sufficient). Let $x\in H^q(K(S^*A))$.
		
		By definition, $\Theta\Psi(e_i)\in \prod_j\on{Hom}((\nabla k)_j, t_{i+j})$, thus \[\Theta\Psi(e_i)^{q'}\in \on{Hom}((\nabla k)_{q'},t_{i+q'})=\on{Hom}(M(q',\dots,q'), \oplus_{\sum i_h=i+q'} M(i_1,\dots,i_p)).\] Now, $E_A\hat{M}$ and the multiplication $A^{\otimes p}\to A$ induce a homomorphism \[f^{i,q'}:\on{Hom}(M(q',\dots,q'), \oplus_{\sum i_h=i+q'} M(i_1,\dots,i_p))\to \on{Hom}((K(S^*A)^{\otimes p})^{i+q'}, S^{q'}A).\] By \Cref{epstein=may}, \[\theta_*(e_i\otimes x^{\otimes p})=(f^{i,q'}(\Theta\Psi(e_i)^{q'}))(x^{\otimes p}),\] where $i+q'=pq$, i.e. $q'=pq-i$. It thus suffices to prove that $\Theta\Psi(e_i)^{q'}=0$, where $q'=pq-i$. By \cite[Theorem 5.1.2]{epstein1966semisimplicial}, this holds when $i>q'(p-1)$. This is equivalent to $i>(pq-i)(p-1)$, that is, $i>q(p-1)$, which is true by assumption.
		
		(b) As in \Cref{cp}, in the case $p\neq 2$, for every integer $n$ we let \[\nu(n):=(-1)^{n(n-1)(p-1)/4}((p-1)/2)!.\] Note that $(((p-1)/2)!)^2 \equiv (-1)^{(p+1)/2}\pmod p$, and so $\nu(n)\nu(-n)\equiv 1\pmod p$. By definition, $\on{P}^0(x)=\nu(-q)D_{q(p-1)}(x)$ when $p>2$, and $\on{Sq}^0(x)=D_{q}(x)$ when $p=2$. 
		
		Assume that $p>2$. Let \[f(q)^q:=\Theta\Psi(e_{q(p-1)})^{q}:M(q,\dots,q)\to \oplus_{\sum i_h=pq} M(i_1,\dots,i_p).\] By \cite[Theorem 5.1.2]{epstein1966semisimplicial}, all the components of $f(q)^q$ are zero, except $M(q,\dots,q)\to M(q,\dots,q)$, which is given by multiplication by an integer congruent to $\nu(q)$ modulo $p$. It follows that \[D_{q(p-1)}(x)=\theta_*(e_i\otimes x^{\otimes p})= \nu(q)S^{q(p-1)}(m^{\otimes p})(x^{\otimes p})= \nu(q)S^{q(p-1)}(\on{Fr})(x)\] for every $x\in H^q(K)$. Here $\on{Fr}:A\to A$ denotes the Frobenius endomorphism of $A$. In the last step, we have used the fact that $\on{Fr}= m^{\otimes p}\circ i$, where $i:A\to A^{\otimes p}$ is defined on sections by $a\mapsto a^{\otimes p}$. Following \cite[8.1.4]{epstein1966steenrod}, $S^{i}$ denotes the $(i+1)$-fold iteration of $S$, that is, the degree $i$ component of $S^*$. Therefore \[D_i(x)=\nu(-q)D_{q(p-1)}(x)=\nu(-q)\nu(q)S^{q(p-1)}(\on{Fr})(x)=S^{q(p-1)}(\on{Fr})(x),\] as desired. A similar argument (with no sign issues) works when $p=2$.
	\end{proof}

	\section{Steenrod operations on \v{C}ech cohomology}\label{cech}

\subsection{The Eilenberg-Zilber operad}

Let $R$ be a commutative ring with identity. 
For every $n\geq 0$, let $\Delta^n$ be the standard $n$-dimensional simplicial set, and denote by $R[\Delta^n]$ the free simplicial $R$-module on $\Delta^n$. We let $\Lambda$ and $\Lambda_N$ be the cosimplicial cochain complexes of $R$-modules such that
 \[\Lambda^n:=K(R[\Delta^n]),\qquad \Lambda_N^n:=K_N(R[\Delta^n])\] for all $n\geq 0$. Here $K(-)$ and $K_N(-)$ are the unnormalized and normalized cochain complex functors, respectively. For all $n\geq 0$, $\Lambda^n$ and $\Lambda_N^n$ are non-positively graded, $\Lambda_N^n$ is bounded, and $\Lambda^n$ is unbounded in the negative direction. 

Let $\mc{T}$ be a topos. If $M$ and $N$ are two cosimplicial cochain complexes of $R$-modules on $\mc{T}$, we define a cochain complex $\underline{\mc{H}om}_{\Delta}(M,N)$ of $R$-modules on $\mc{T}$ as the equalizer of \[\prod_{r\geq 0}\underline{\mc{H}om}(M^r,N^r)\rightrightarrows \prod_{[r]\to [s]} \underline{\mc{H}om}(M^r,N^s),\] where the second product is over all arrows in the simplicial category $\Delta$, and the two maps are induced by pre-composition and post-composition. This definition is a special case of \cite[(2.1)]{may2003operads}. When $\mc{T}=\mc{S}et$, so that $M$ and $N$ are cosimplicial cochain complexes of $R$-modules, we denote $\underline{\mc{H}om}_{\Delta}(M,N)$ by $\underline{\on{Hom}}_{\Delta}(M,N)$. 

Recall that we write $(e_*,e^{-1}):\mc{T}\to \mc{S}et$ for the morphism of topoi such that $e_*=\Gamma(\mc{T},-)$ and $e^{-1}$ is the constant sheaf functor.
\begin{lemma}\label{totalization}
	Let $C$ be a cosimplicial non-negatively graded cochain complex of $R$-modules on $\mc{T}$. Then we have a commutative diagram
	\[
	\begin{tikzcd}
	\on{Tot}(K_N(C)) \arrow[r] \arrow[d,"\wr"] & \on{Tot}(K(C)) \arrow[d,"\wr"] \\
	\underline{\mc{H}om}_{\Delta}(e^{-1}\Lambda_N,C) \arrow[r] & \underline{\mc{H}om}_{\Delta}(e^{-1}\Lambda,C), 
	\end{tikzcd}
	\]
	where the vertical arrows are isomorphisms, and the horizontal arrows are cochain homotopy equivalences.
\end{lemma}

\begin{proof}
	The cochain complex $\underline{\mc{H}om}_{\Delta}(e^{-1}\Lambda,C)$ is given in degree $n$ by the equalizer of
	\[\prod_{r\geq 0}\prod_{q\in \Z} \mc{H}om(e^{-1}K\Z[\Delta^r]^q,(C^r)^{q+n})\rightrightarrows \prod_{[r]\to [s]}\prod_{q\in \Z}\mc{H}om(e^{-1}K\Z[\Delta^r]^q,(C^s)^{q+n}).\] 
	If $M$ is an $R$-module on $\mc{T}$, then $\mc{H}om (e^{-1}R,M)\cong M$. It follows that \[\mc{H}om(e^{-1}K\Z[\Delta^r]^q,(C^s)^{q+n})\cong ((C^s)^{q+n})^{\oplus \Delta^r_{-q}}\] is a direct sum of copies of $(C^s)^{q+n}$, parametrized by the set $\Delta^r_{-q}$ of $(-q)$-dimensional simplices of $\Delta^r$. Thus $\underline{\mc{H}om}_{\Delta}(e^{-1}\Lambda,C)$ in degree $n$ is the equalizer of
	\[\prod_{r\geq 0}\prod_{q\in \Z} ((C^r)^{q+n})^{\oplus \Delta^r_{-q}}\rightrightarrows \prod_{[r]\to [s]}\prod_{q\in \Z}((C^s)^{q+n})^{\oplus \Delta^r_{-q}}.\]
	We now construct the map $\on{Tot}(K(C))\to \underline{\mc{H}om}_{\Delta}(e^{-1}\Lambda,C)$ in degree $n$. For every integer $0\leq i\leq n$, let $\sigma_i\in \Delta^i_i$ be the fundamental class. If $r$ is another integer and $\sigma\in \Delta^r_i$, we define $(C^{i})^{n-i}\to ((C^r)^{n-i})^{\oplus \Delta^r_{i}}$ on the component relative to $\sigma$ as the homomorphism induced by the map $[i]\to [r]$ given by $\sigma_i\mapsto \sigma$.
	We set $(C^{i})^{n-i}\to ((C^r)^{q+n})^{\oplus \Delta^r_{-q}}$ equal to zero when $q\neq -i$. These maps assemble to a homomorphism from $\on{Tot}(K(C))$ to the product on the left which respects the equalizer condition. We thus get a homomorphism \[\on{Tot}K(C)\to \underline{\mc{H}om}_{\Delta}(e^{-1}\Lambda,C).\] It is an isomorphism, whose inverse is induced by projection onto the factors corresponding to the fundamental classes of the $\Delta^i_i$. The proof for $K_N(C)$ is entirely analogous, with the difference that one only works with non-degenerate simplices. The commutativity of the diagram is then obvious. 
	
	If $A$ is a cosimplicial object of an abelian category, the map $K_N(A)\to K(A)$ is a homotopy equivalence. Applying this to the category of cochain complexes of $R$-modules on $\mc{T}$, we see that $K_N(C)\to K(C)$ is a homotopy equivalence of double complexes (that is, complexes of complexes). Passing to total complexes, we obtain the required cochain homotopy equivalence $\on{Tot}(K_N(C))\to \on{Tot}(K(C))$. 
\end{proof}

 We define $\mc{Z}:=\on{End}(\Lambda)$, the endomorphism operad of $\Lambda$ in the category of cochain complexes of $R$-modules; see \cite[Construction 2.3]{may2003operads}. Similarly, we define $\mc{Z}_N:=\on{End}(\Lambda_N)$. 
By definition,  for every $j\geq 0$ we have cochain complexes \[\mc{Z}(j)=\on{Hom}_{\Delta}(\Lambda,\Lambda^{\otimes j}),\qquad \mc{Z}_N(j)=\on{Hom}_{\Delta}(\Lambda_N,\Lambda_N^{\otimes j})\] which are concentrated in non-positive degrees.

\begin{lemma}\label{norm-unnorm-qi}
	For every $j\geq 0$, we have a homotopy equivalence $\mc{Z}_N(j)\to \mc{Z}(j)$. 
\end{lemma}
The maps of \Cref{norm-unnorm-qi} assemble to a quasi-isomorphism of operads $\mc{Z}_N\to \mc{Z}$, but we will not need this.
\begin{proof}
	By \Cref{totalization}, it is enough to show the composition
	\begin{equation}\label{norm-unnorm}
	\on{Tot}(K_N(\Lambda_N^{\otimes j}))\to \on{Tot}(K(\Lambda_N^{\otimes j}))\to \on{Tot}(K(\Lambda^{\otimes j})).
	\end{equation}
	is a homotopy equivalence.
	
	If $A$ is a cosimplicial object in an abelian category $\mc{A}$, by the Dold-Kan correspondence we have a decomposition $K(A)=K_N(A)\oplus D(A)$ in the category of cochain complexes of $\mc{A}$; see \cite[Tag 019I]{stacks-project} or \cite[Theorem III.2.5]{goerss2009simplicial}. The complex $D(A)$ is homotopically equivalent to zero, and so the inclusion $K(A)\to K_N(A)$ is a cochain homotopy equivalence, natural in $A$. The cochain homotopy between the identity on $K(A)$ and the composition $K(A)\to K_N(A)\to K(A)$ is also natural in $A$; see the paragraph preceding \cite[Theorem 2.4]{goerss2009simplicial}.
	
	We apply this to the case when $\mc{A}$ is the category of cochain complexes of $R$-modules, and then pass to total complexes. If a bicomplex is homotopy equivalent to zero (as a complex in the category of complexes), then its totalization is also homotopy equivalent to zero. Thus, if $C$ is a cosimplicial $R$-module, we have a decomposition $\on{Tot}K(C)=\on{Tot}K_N(C)\oplus \on{Tot}D(C)$, and  $\on{Tot}D(C)$ is homotopically equivalent to zero. Letting $C=\Lambda_N^{\otimes j}$, we deduce that the first map in (\ref{norm-unnorm}) is a homotopy equivalence. 
	
	We have a decomposition $\Lambda=\Lambda_N\oplus \Lambda'$. (For every $n\geq 0$, $(\Lambda')^n$ is the subcomplex of $K(R[\Delta^n])$ generated by degenerate simplices.) Thus $\Lambda^{\otimes j}=\Lambda_N^{\otimes j}\oplus U$, where $U$ is a direct sum of terms of the form $U_1\otimes\cdots\otimes U_j$, where $U_h\in \set{\Lambda_N,\Lambda'}$ and at least one of the $U_h$ equals $\Lambda'$. Passing to total complexes, we see that in order to prove that the second map is a homotopy equivalence, it suffices to show that each $K(U_1\otimes\cdots\otimes U_j)$ is homotopy equivalent to zero (as a complex of complexes), for then $\on{Tot}(K(U_1\otimes\cdots\otimes U_j))$ will also be homotopically trivial. The classical Alexander-Whitney map gives a homotopy equivalence between $K(U_1\otimes\cdots\otimes U_j)$ and $K(U_1)\otimes\cdots \otimes K(U_j)$, so it suffices to show that the latter are homotopically trivial. 
	
	Since the homotopies in the Dold-Kan correspondence are functorial, $K(\Lambda')$ is homotopically trivial (as a complex in the category of complexes). Let $M$, $N$ and $N'$ be double complexes, and let $f:N\to N'$ be a homotopy equivalence in the category of complexes of complexes: $f^n:N^n\to (N')^n$ is a homotopy equivalence for every $n$, and the homotopies commute with vertical differentials. Then $\on{id}\otimes f:M\otimes N\to M\otimes N'$ is also a homotopy equivalence (here $\otimes$ is the tensor product in the category of complexes of complexes). Since $K(\Lambda')$ is homotopically trivial and at least one $U_h$ is equal to $\Lambda'$,  we conclude that each $K(U_1)\otimes\cdots\otimes K(U_j) $ is homotopically equivalent to zero (as a complex of complexes), as desired. We conclude that the $K(U_1)\otimes\cdots \otimes K(U_j)$ are homotopically trivial, and so the second map of (\ref{norm-unnorm}) is also a homotopy equivalence.
\end{proof}

Note that the operad defined in \cite[Definition 3.1]{may2003operads} is denoted there by $\mc{Z}$,  but it coincides with our $\mc{Z}_N$. By a result of Mandell \cite[Proposition 3.2]{may2003operads}, for all $j\geq 0$ there is an augmentation map $\mc{Z}_N(j)\to R$ which is a quasi-isomorphism. By \cite[Proposition 3.3]{may2003operads}, there are an $E_{\infty}$-operad $\mc{E}$ and a quasi-isomorphism of operads $\alpha_N:\mc{E}\to \mc{Z}_N$. In particular, $\mc{E}(j)$ is an $R[\Sigma_j]$-free resolution of $R$ for every $j\geq 0$, concentrated in non-positive degrees, and we have a $\Sigma_j$-equivariant quasi-isomorphism of cochain complexes $\alpha_N(j):\mc{E}(j)\to \mc{Z}_N(j)$.

Pre-composing with $\alpha_N(j)$, we obtain a $\Sigma_j$-equivariant quasi-isomorphism
\[\alpha(j):\mc{E}(j)\to \mc{Z}(j).\]
Since $W(p,R)$ and $\mc{E}(p)$ are both $R[\pi]$-free resolutions of $R$, by \cite[Tag 0649]{stacks-project} there exists a $\pi$-homotopy equivalence $W(p,R)\to \mc{E}(p)$ commuting with the augmentations. The composition
\begin{equation}\label{w-e-z}
W(p,R)\to \mc{E}(p)\xrightarrow{\alpha(p)}\mc{Z}(p)
\end{equation}
is a $\pi$-equivariant quasi-isomorphism. Since $e^{-1}$ is exact (it commutes with colimits and finite limits), $e^{-1}\alpha(p)$ and the pullback of (\ref{w-e-z}) are also quasi-isomorphisms.

	\subsection{The relative \v{C}ech complex}
In \cite{may2003operads}, May defined an $E_{\infty}$-algebra structure on \v{C}ech cochains, thus obtaining Steenrod operations on the \v{C}ech cohomology of sheaves. We will use this alternative definition to construct the Bockstein homomorphism and during the proof of \Cref{mainthm}(a). Note that May works in the category of sheaves on a topological space, but, as we explain below, his arguments may be easily adapted to the setting of sheaves on a site; this is explicitly mentioned in \cite[\S 4, p. 9 and Remark 5.10]{may2003operads}. 

Let $\mc{T}$ be a topos, let $R$ be a commutative ring with identity, and let $F$ be a cochain complex of $R$-modules on $\mc{T}$. We fix an equivalence of $\mc{T}$ with the category of sheaves on a site $\mc{S}$ admitting a terminal object; by Giraud's criterion \cite[IV, Th\'eor\`eme 1.2(i')]{sga4I} such $\mc{S}$ always exists (it is usually straightforward to exhibit a concrete $\mc{S}$ in practice). We view $F$ as a complex of sheaves of $R$-modules on $\mc{S}$ via this equivalence. 

Let $e$ be a terminal object of $\mc{S}$, and let $U$ be the \v{C}ech nerve (that is, the $0$-coskeleton) of a cover of $e$ in $\mc{S}$. For every $r\geq 0$ let $\eta_r:U_r\to e$ denote the unique map to $e$, and let $(\eta_{r*},\eta_r^{-1}):\mc{T}/U_r\to \mc{T}$ be the associated morphism of topoi. We define $\check{\mc{C}}_0(U,F)$ as the cosimplicial complex of sheaves on $\mc{S}$ such that \[\check{\mc{C}}_0(U,F)^{r,s}:=\eta_{r*}(\eta_r^{-1}(F^s)),\] and whose differentials are induced from those of $F$ and the degeneracy maps of $U$. Similarly, we denote by $\check{\mc{C}}(U,F):=K(\check{\mc{C}}_0(U,F))$ the bicomplex of sheaves on $\mc{S}$ such that \[\check{\mc{C}}(U,F)^{r,s}:=\eta_{r*}(\eta_r^{-1}(F^s)),\] and whose differentials are induced from those of $F$ and the alternating sums of the face maps of $U$. We set \[\check{C}_0(U,F):=\Gamma(\mc{T},\check{\mc{C}}_0(U,F)),\qquad \check{C}(U,F):=\Gamma(\mc{T},\check{\mc{C}}(U,F)),\] the global section functor being applied level-wise. Thus, for all $r,s\geq 0$,
\[\check{C}(U,F)^{r,s}=\Gamma(\mc{T},\eta_{r*}(\eta_r^{-1}(F^s)))=\Gamma(\mc{T}/U_r,\eta_r^{-1}(F^s))=\Gamma(U_r,F^s).\]
For every $s\geq 0$, the cochain complex $\check{\mc{C}}(U,F)^{*,s}$ is the relative \v{C}ech complex of $F^s$; see \cite[Tag 06X7]{stacks-project}. The unit maps $F^s\to \eta_{0*}(\eta_0^{-1}(F^s))$ induce a homomorphism of bicomplexes $F\to \check{\mc{C}}(U,F)$ which is natural in $F$; here $F$ is regarded as a bicomplex with exactly one non-zero row. Passing to total complexes, we obtain a homomorphism of complexes of sheaves \[\iota_F: F\to \on{Tot}\check{\mc{C}}(U,F).\] By \cite[Lemma 2.4.18]{olsson2016algebraic}, for every $s$ the map $F^s\to \check{\mc{C}}(U,F)^{*,s}$ is a quasi-isomorphism, hence $\iota_F$ is a quasi-isomorphism.

If $F'$ is another cochain complex of $R$-modules on $\mc{T}$, we have a natural homomorphism
\begin{equation}\label{tot-is-lax}
\cup:\on{Tot}\check{\mc{C}}(U,F)\otimes \on{Tot}\check{\mc{C}}(U,F')\to \on{Tot}\check{\mc{C}}(U,F\otimes F')
\end{equation}
The map (\ref{tot-is-lax}) is compatible with $\iota_F,\iota_{F'}$ and $\iota_{F\otimes F'}$ in an obvious way. 

Let $j\geq 0$ be an integer, and let $C$ be a cosimplicial complex of $R$-modules on $\mc{T}$. The Alexander-Whitney map is a homomorphism 
\begin{equation}\label{alex-whitney}e^{-1}\mc{Z}(j)\otimes (\on{Tot}K(C))^{\otimes j}\to \on{Tot}K(C^{\otimes j}),\end{equation}
where on the right we are taking tensor products of cosimplicial objects. Under the identifications provided by \Cref{totalization}, (\ref{alex-whitney}) is given by the composition
\begin{align*}
\underline{\mc{H}om}_{\Delta}(e^{-1}\Lambda,e^{-1}\Lambda^{\otimes j})\otimes \underline{\mc{H}om}_{\Delta}(e^{-1}\Lambda&,C)^{\otimes j} \\ \to \underline{\mc{H}om}_{\Delta}(e^{-1}\Lambda,e^{-1}\Lambda^{\otimes j})&\otimes \underline{\mc{H}om}_{\Delta}(e^{-1}\Lambda^{\otimes j},C^{\otimes j})\\ \to  \underline{\mc{H}om}_{\Delta}&(e^{-1}\Lambda^{\otimes j},C^{\otimes j})\otimes \underline{\mc{H}om}_{\Delta}(e^{-1}\Lambda,e^{-1}\Lambda^{\otimes j})\\ &\to \underline{\mc{H}om}_{\Delta}(e^{-1}\Lambda,C^{\otimes j}),
\end{align*}
where the first map is induced by the $j$-fold tensor product, the second is the (graded) switch isomorphism, and the third is induced by composition. This is a special case of \cite[Construction 2.6]{may2003operads} (to see this, note that $e^{-1}$ commutes with finite limits, and so $e^{-1}\on{End}(\Lambda)(j)=\on{End}(e^{-1}\Lambda)(j)$).
Of course, if $\mc{T}$ is the set topos, so that $C$ is just a cosimplicial cochain complex of $R$-modules, (\ref{alex-whitney}) takes the form
\begin{equation}\label{alex-whitney-set}\mc{Z}(j)\otimes (\on{Tot}K(C))^{\otimes j}\to \on{Tot}K(C^{\otimes j}).\end{equation}

\subsection{May's construction}

	Let $A$ be an commutative differential graded $R$-algebra on $\mc{T}$. 
	We have a $\Sigma_p$-equivariant homomorphism\footnote{In \cite[Theorem 5.5]{may2003operads} a similar map is constructed, with $\mc{Z}$ replaced by $\mc{Z}_N$ and $\check{C}(U,A)$ replaced by the totalization of the restricted \v{C}ech complex. The homomorphism (\ref{zp}) is the reason why we use $\mc{Z}$ instead of $\mc{Z}_N$. The operad $\mc{Z}_N$ only acts on the restricted \v{C}ech cochains (see \cite[(4.2)]{may2003operads}), which are only useful when the covers are monomorphisms (e.g. open embeddings); see e.g. \cite[III, Remark 2.2(d)]{milne1980etale}.}
	\begin{equation}\label{zp}\mc{Z}(p)\otimes \on{Tot}\check{C}(U,A)^{\otimes p}\to \on{Tot}\check{C}(U,A).\end{equation}
	Under the identifications of \Cref{totalization}, (\ref{zp}) is defined as the composition
\begin{align*}
\underline{\on{Hom}}_{\Delta}(\Lambda,\Lambda^{\otimes p})\otimes \underline{\on{Hom}}_{\Delta}(\Lambda,\check{C}_0(U,A))^{\otimes p}& \\ \to \underline{\on{Hom}}_{\Delta}(\Lambda, \check{C}_0&(U,A)^{\otimes p}) \\ \to \underline{\on{Hom}}&_{\Delta}(\Lambda, \check{C}_0(U,A^{\otimes p})) \\ &\to \underline{\on{Hom}}_{\Delta}(\Lambda, \check{C}_0(U,A)),
\end{align*}	
	where the first homomorphism is the Alexander-Whitney map, the second homomorphism is induced by the natural map $\check{C}_0(U,A)^{\otimes p}\to \check{C}_0(U,A^{\otimes p})$, and the third homomorphism is induced by the multiplication map $A^{\otimes p}\to A$. 
	
	\begin{rmk}\label{remark-functorial}
	It is clear that (\ref{alex-whitney}) is functorial in $C$, and that (\ref{zp}) is functorial in $A$. 
	
	Let $\mc{T}'$ be another topos, let $e'$ be a terminal object in some site with associated topos equivalent to $\mc{T}'$, and let $\eta':U'\to e'$ be the \v{C}ech resolution of a cover of $e'$. Let $\phi:\check{C}_0(U,R)\to \check{C}_0(U',R)$ be a map of cosimplicial (cochain) complexes, such that for every $n\geq 0$ the map $\Gamma(\mc{T},\eta_{n*}\eta_n^{-1}R)\to \Gamma(\mc{T}',\eta'_{n*}(\eta_n')^{-1}R)$ is a homomorphism of $R$-algebras. Since (\ref{alex-whitney-set}) is functorial in $C$, it is compatible with $\phi$. Moreover, thanks to the additional assumptions on $\phi$, the second and third map in the definition of (\ref{zp}) are also compatible with $\phi$. In other words, we have a commutative diagram
	\[
	\begin{tikzcd}
	\mc{Z}(p)\otimes \on{Tot}\check{C}(U,R)^{\otimes p}\arrow[r] \arrow[d] &  \on{Tot}\check{C}(U,R)\arrow[d] \\
	\mc{Z}(p)\otimes \on{Tot}\check{C}(U',R)^{\otimes p}\arrow[r] &  \on{Tot}\check{C}(U',R),
	\end{tikzcd}
	\]
	where the horizontal arrows are the Alexander-Whitney maps, and the vertical maps are induced by $\phi$. We will use this remark in the proof of \Cref{mainthm}(a).
	\end{rmk}
	
		We now show that (\ref{zp}) is induced by a morphism of sheaves. 
	
	\begin{lemma}
		There exists a morphism of complexes of sheaves of $R$-modules \[e^{-1}\mc{Z}(p)\otimes \on{Tot}\check{\mc{C}}(U,A)^{\otimes p}\to \on{Tot}\check{\mc{C}}(U,A)\] such that the composition 
		\[\mc{Z}(p)\otimes \on{Tot}\check{C}(U,A)^{\otimes p}\to \Gamma(\mc{T},e^{-1}\mc{Z}(p))\otimes \Gamma(\mc{T},\on{Tot}\check{\mc{C}}(U,A)^{\otimes p})\to \on{Tot}\check{C}(U,A)\] obtained by passage to global sections coincides with (\ref{zp}).
	\end{lemma}

\begin{proof}
	Since $e^{-1}$ and equalizers are limits, they commute with each other, and so we have \[e^{-1}\underline{\on{Hom}}_{\Delta}(M,N)\cong \underline{\mc{H}om}_{\Delta}(e^{-1}M,e^{-1}N)\] for any two cosimplicial cochain complexes $M$ and $N$. Recall also that $e^{-1}$ commutes with tensor products. We define the following composition:
	\begin{align*}
	\underline{\mc{H}om}_{\Delta}(e^{-1}\Lambda,e^{-1}\Lambda^{\otimes p})\otimes \underline{\mc{H}om}_{\Delta}(e^{-1}\Lambda,\check{\mc{C}}_0(U,A))^{\otimes p}&\\ \to \underline{\mc{H}om}_{\Delta}(e^{-1}\Lambda, \check{\mc{C}}_0&(U,A)^{\otimes p}) \\ \to \underline{\mc{H}om}_{\Delta}&(e^{-1}\Lambda, \check{\mc{C}}_0(U,A^{\otimes p})) \\ &\to \underline{\mc{H}om}_{\Delta}(e^{-1}\Lambda, \check{\mc{C}}_0(U,A)).
	\end{align*}
	Here the first homomorphism is the Alexander-Whitney map of \cite[Construction 2.6]{may2003operads}, this time with the category of cochain complexes of $R$-modules on $\mc{T}$ as target category. The proof of \cite[Proposition 4.4]{may2003operads} immediately adapts to give a map $\check{\mc{C}}_0(U,A)^{\otimes p}\to \check{\mc{C}}_0(U,A^{\otimes p})$, which in turn induces the second homomorphism in the composition above. The third homomorphism is induced by the multiplication $A^{\otimes p}\to A$. It is clear that passing to global sections yields the composition defining (\ref{zp}).
\end{proof}

Pre-composing with the pullback of (\ref{w-e-z}), we obtain a $\pi$-equivariant map 
	\begin{equation}\label{sheafy-theta}
	e^{-1}W(p,R)\otimes \on{Tot}\check{\mc{C}}(U,A)^{\otimes p}\to \on{Tot}\check{\mc{C}}(U,A).
	\end{equation}
	Taking global sections, and pre-composing with the adjunction unit and the natural map $\on{Tot}\check{C}(U,A)^{\otimes p}\to \Gamma(\mc{T},\on{Tot}\check{\mc{C}}(U,A))^{\otimes p}$, we finally get
	\begin{equation}\label{check-global}\check{\theta}:W(p,R)\otimes \on{Tot}\check{C}(U,A)^{\otimes p}\to \on{Tot}\check{C}(U,A).\end{equation} The map $\check{\theta}$ is uniquely defined up to $\pi$-homotopy equivalence. It is compatible with base change along homomorphisms of rings $R\to S$ in an obvious way.

\begin{lemma}
	Let $A$ be an $R$-algebra on $\mc{T}$. Then the pair $(\on{Tot}\check{C}(U,A),\check{\theta})$ construced above is an object of $\mc{C}(p,R)$. 
\end{lemma}

\begin{proof}
	We have to check properties (i) and (ii) of \Cref{cp}. 
	
	(i) By \Cref{totalization}, we see that $\mc{Z}(p)^{0}$ is the cochain complex
	\[R\xrightarrow{0} R\xrightarrow{\on{id}} R\xrightarrow{0}R\to\cdots,\]
	where the first copy of $R$ is in degree $0$. The Alexander-Whitney map is the identity in the $0$-th row. It is now easy to check that (\ref{zp}) is the map \[\Gamma(\mc{T},\eta_{0*}\eta_0^{-1}A)^{\otimes p}\to \Gamma(\mc{T},\eta_{0*}\eta_0^{-1}A)\] obtained from the multiplication $A^{\otimes p}\to A$ by applying the adjunction unit and then taking global sections. Composing with $\alpha(p)$ and $W(p,R)\to \mc{E}(p)$, we deduce (i).
	
	(ii) We may take $V=\mc{E}(p)$ and $j$ equal to the map $W(p,R)\to \mc{E}(p)$ appearing in (\ref{w-e-z}). 
\end{proof}

\begin{lemma}
	We have a $\pi$-homotopy commutative square
\[
\begin{tikzcd}
\check{\mc{C}}(U,A)^{\otimes p} \arrow[r] & \underline{\mc{H}om}(e^{-1}W(p,R),\check{\mc{C}}(U,A))\\
A^{\otimes p} \arrow[r,"m_p"] \arrow[u,"\iota_A^{\otimes p}"]  & A \arrow[u]
\end{tikzcd}
\]
where the top horizontal map is adjoint to (\ref{sheafy-theta}). 
\end{lemma}

\begin{proof}
	If we regard $A$ as a cosimplicial cochain complex which is zero in positive levels (so all maps in the cosimplicial direction are zero), it is easy to check that the Alexander-Whitney map (\ref{alex-whitney}) for $F=A$ is just $e^{-1}(\epsilon)\otimes\on{id}$, where $\epsilon:\mc{Z}(p)\to R$ is the augmentation map.
	
	By the functoriality and equivariance of the Alexander-Whitney map, we then have a $\pi$-equivariant commutative diagram
	\[
	\begin{tikzcd}
	e^{-1}\mc{Z}(p)\otimes \on{Tot}\check{\mc{C}}(U,A)^{\otimes p} \arrow[r] & \on{Tot}\check{\mc{C}}(U,A)^{\otimes p} \arrow[r] &  \on{Tot}\check{\mc{C}}(U,A)\\
	e^{-1}\mc{Z}(p)\otimes A^{\otimes p} \arrow[r,"e^{-1}(\epsilon)\otimes\on{id}"] \arrow[u,"\on{id}\otimes \iota_A^{\otimes p}"] & A^{\otimes p} \arrow[u,"\iota_A^{\otimes p}"]  \arrow[r,"m_p"]   & A. \arrow[u,"\iota_A"]
	\end{tikzcd}
	\]

	By the tensor-hom adjunction, we obtain a $\pi$-equivariant commutative diagram
	\[
	\begin{tikzcd}
	\on{Tot}\check{\mc{C}}(U,A)^{\otimes p} \arrow[rr] && \underline{\mc{H}om}(e^{-1}\mc{Z}(p),\on{Tot}\check{\mc{C}}(U,A))\\
	A^{\otimes p} \arrow[r,"m_p"] \arrow[u,"\iota_A"] & A \arrow[r] & \underline{\mc{H}om}(e^{-1}\mc{Z}(p),A) \arrow[u,"\iota_A^*"]  
	\end{tikzcd}
	\]
	We obtain the conclusion by composing with the map \[\underline{\mc{H}om}(e^{-1}\mc{Z}(p),\on{Tot}\check{\mc{C}}(U,A))\to \underline{\mc{H}om}(e^{-1}W(p,R),\on{Tot}\check{\mc{C}}(U,A))\] induced by (\ref{w-e-z}).
\end{proof}

\subsection{Relation with derived functor cohomology}
Let $\nu:A\to I$ be a standard injective resolution by $R$-modules. We define a homomorphism 
\begin{equation}\label{cech-map}H^*(\on{Tot}\check{C}(U,A))\to \mathbb{H}^*(\mc{T},A).\end{equation}
as follows. The resolution $\nu$ induces a homomorphism $\nu^*:\on{Tot}\check{C}(U,A)\to \on{Tot}\check{C}(U,I)$. The map $\iota_I$ induces a homomorphism $\Gamma(\iota_I):\Gamma(\mc{T},I)\to \on{Tot}\check{C}(U,I)$.

\begin{lemma}\label{inj-qiso}
	The map $\Gamma(\iota_I)$ is a quasi-isomorphism.
\end{lemma}

\begin{proof}
	By \cite[Tag 03AW]{stacks-project}, we have $\check{H}^i(U,I^q)=0$ for all $i>0$ and $q\geq 0$. Here $\check{H}^i(U,-)$ denotes \v{C}ech cohomology with respect to $U$. The conclusion now follows from the application of \cite[Tag 0133]{stacks-project} to the double complex $\check{C}(U,I)$.
\end{proof}

Thus we have a morphism $\Gamma(\iota_I)^{-1}\circ \nu^*$ in the derived category, and we define (\ref{cech-map}) as the induced homomorphism in cohomology.  

The next lemma shows that (\ref{cech-map}) comes from a homomorphism of sheaves.
\begin{lemma}\label{comes-from-sheaf}
	There is a homomorphism $h:\on{Tot}\check{\mc{C}}(U,A)\to I$ such that $h\circ\iota_A=\nu$ and the map induced in cohomology by $\Gamma(\mc{T},h)$ is (\ref{cech-map}).
\end{lemma}

\begin{proof}
By naturality, we have a commutative diagram of solid arrows
\begin{equation}\label{cech-inj}
\begin{tikzcd}
A\arrow[r,"\iota_A"]  \arrow[d,"\nu"] & \on{Tot}\check{\mc{C}}(U,A) \arrow[d,"\xi"] \arrow[dl,dotted,swap,"h"]  \\   
I \arrow[r,"\iota_I"] & \on{Tot}\check{\mc{C}}(U,I).
\end{tikzcd}
\end{equation}
Since $\nu$, $\iota_A$ and $\iota_I$ are quasi-isomorphisms, so is $\xi$. It follows that every solid arrow in (\ref{cech-inj}) is an isomorphism in the derived category.  Note that $\Gamma(\xi)$ is the map $\nu^*$ defined in the previous paragraph.  Since $I$ is injective in every degree, there exists a cochain map \[\rho:\on{Tot}\check{\mc{C}}(U,I)\to I\] such that $\rho\circ \iota_I$ is homotopic to the identity on $I$ (apply \cite[Tag 05TG]{stacks-project} to the map $\iota_I^{-1}$ in the derived category). Define $h:=\rho\circ\xi$. Then by \cite[Tag 05TG]{stacks-project} the top triangle of (\ref{cech-inj}) is homotopy commutative. Thus we have a diagram
\begin{equation}\label{cech-inj'}
\begin{tikzcd}
A\arrow[r,"\iota_A"]  \arrow[d] & \on{Tot}\check{\mc{C}}(U,A) \arrow[d,"\xi"] \arrow[dl,dotted,swap,"h"]  \\   
I  & \on{Tot}\check{\mc{C}}(U,I).\arrow[l,"\rho"]
\end{tikzcd}
\end{equation}
where the bottom triangle is commutative, and the top triangle is commutative in the derived category, and so is homotopy commutative by \cite[Tag 013S]{stacks-project}. 

We pass to global sections level-wise in (\ref{cech-inj'}):
\begin{equation}\label{cech-inj-glob}
\begin{tikzcd}
\Gamma(\mc{T},A)\arrow[r,"\Gamma(\iota_A)"]  \arrow[d,swap,"\Gamma(\nu)"] & \on{Tot}\check{C}(U,A) \arrow[d,"\Gamma(\xi)"] \arrow[dl,dotted,swap,"\Gamma(h)"]  \\   
\Gamma(\mc{T},I) & \on{Tot}\check{C}(U,I).\arrow[l,swap,"\Gamma(\rho)"] 
\end{tikzcd}
\end{equation}
Note that taking global sections level-wise respects homotopy equivalences, therefore the bottom triangle in (\ref{cech-inj-glob}) is commutative, and the top triangle is homotopy commutative. By \Cref{inj-qiso}, the map $\Gamma(\iota_I)$ is a quasi-isomorphism and so is an isomorphism in the derived category. Since $\Gamma(\rho)\circ \Gamma(\iota_I)$ is homotopic to the identity on $\Gamma(\mc{T},I)$, we deduce that $\Gamma(\rho)$ is a left inverse of $\Gamma(\iota_I)$ in the derived category, and since $\Gamma(\iota_I)$ is an isomorphism in the derived category, we obtain that $\Gamma(\rho)=\Gamma(\iota_I)^{-1}$ in the derived category. 
It follows that $\Gamma(h)$ coincides with $\Gamma(\iota_I)^{-1}\circ \Gamma(\xi)$ in the derived category of cochain complexes. The latter is exactly the map of \cite[Tag 08BN]{stacks-project}. Since $\Gamma(h)$ and $\Gamma(\iota_I)^{-1}\circ\Gamma(\xi)$ coincide in the derived category, they induce the same homomorphism in cohomology, as desired.
\end{proof}

Denote by $\on{PSh}(\mc{S})$ the category of presheaves on $\mc{S}$. Following \cite[Expos\'e V, 2.4.2]{sga4I}, for all $i\geq 0$ we denote by $\mc{H}^i(S,-)$ the right derived functors of the forgetful functor $\mc{T}\to \on{PSh}(\mc{S})$. If $S$ is an object of $\mc{S}$, we denote by $H^i(S,-)$ the right derived functors of $\Gamma(S,-):\mc{T}\to \mc{S}et$. For all sheaves $F$ on $\mc{S}$, we have $\mc{H}^i(S,F)=H^i(S,F)=H^i(\mc{T}/S,F)$.
	
\begin{prop}\label{cech-to-derived}	
	Assume further that $R=\F_p$ and that $H^i(U_n,A^q)=0$ for every $i>0$ and every $n,q\geq 0$. Then (\ref{cech-map}) is an isomorphism and is compatible with Steenrod operations.
\end{prop}	

\begin{proof}
Since we do not have a reference for it, we derive the Cartan-Leray spectral sequence for $A$; the original source \cite[Expos\'e V, Th\'eor\`eme 3.2]{sga4I} only applies to the case when $A$ is concentrated in degree $0$. Let $A\to J$ be a Cartan-Eilenberg resolution; see \cite[Tag 015H]{stacks-project} for the definition. We may then let $I=\on{Tot}J$ and $\nu:A\to I$ be the standard injective resolution induced by totalization. Applying \cite[08BI]{stacks-project} twice, we see that \[\on{Tot}(\check{C}(U,I))=\on{Tot}(\check{C}(U,J))=\on{Tot}(\check{C}(U,\on{Tot}(K))),\] where the $\on{Tot}(-)$ in the middle is the totalization of a triple complex, and by definition $K$ is the double complex with terms \[K^{r,s}:=\oplus_{a+b=r}\Gamma(U_a,J^{b,s})\] and maps induced by those of $J$ and $U$. On the other hand, by \Cref{inj-qiso} we have a quasi-isomorphism \[\Gamma(\iota_I):R\Gamma(\mc{T},A)=\Gamma(\mc{T},I)\to \on{Tot}(\check{C}(U,I)).\] Thus the spectral sequence associated to $K$ (see \cite[Tags 0130, 0132]{stacks-project}) reads:
\begin{equation}\label{cech-sseq}E_2^{r,s}:=H^r(\on{Tot}(\check{C}(U,\mc{H}^s(A))))\Rightarrow H^{r+s}(\mc{T},A).\end{equation} 
Here $\mc{H}^s(A)$ is defined as the presheaf sending $V\mapsto H^s(V,A)$. It is easy to see that the edge maps on the bottom horizontal row of (\ref{cech-sseq}) coincide with (\ref{cech-map}). Under the assumptions of the lemma, the $E_2$-page of (\ref{cech-sseq}) is concentrated in the bottom row, hence (\ref{cech-map}) is an isomorphism.

	We have the following diagram, where the two smaller triangles and rectangles are $\pi$-homotopy commutative:
	\[
	\begin{tikzcd}
	\on{Tot}\check{\mc{C}}(U,A)^{\otimes p}\arrow[rrr] \arrow[dd, "h^{\otimes p}"] &&& \underline{\mc{H}om}_{\F_p}(e^{-1}W,\on{Tot}\check{\mc{C}}(U,A)) \arrow[dd, "h^*"] \\
	& A^{\otimes p} \arrow[ul,swap,"\iota_A^{\otimes p}"] \arrow[dl] \arrow[r]  & A \arrow[ur] \arrow[dr]   & \\
	I^{\otimes p} \arrow[rrr]  &&&  \underline{\mc{H}om}_{\F_p}(e^{-1}W,I).	
	\end{tikzcd}
	\]
	Since $\iota_A$ is a quasi-isomorphism of complexes over a field, $\iota_A^{\otimes p}$ is also a quasi-isomorphism. A diagram chase now shows that the outer square is commutative in the derived category. Since every term of $\underline{\mc{H}om}(W,I)$ is $\F_p[\pi]$-injective, it follows from \cite[Tag 05TG]{stacks-project} that the outer square is $\pi$-homotopy commutative. By the tensor-hom adjunction, we obtain a $\pi$-homotopy commutative square
	\[
	\begin{tikzcd}
	e^{-1}W\otimes \on{Tot}\check{\mc{C}}(U,A)^{\otimes p} \arrow[r,"\check{\theta}"] \arrow[d, "\on{id}\otimes h^{\otimes p}"]  &  \on{Tot}\check{\mc{C}}(U,A) \arrow[d, "h"]  \\
	e^{-1}W\otimes I^{\otimes p} \arrow[r,"\theta"] & I. 
	\end{tikzcd}	
	\]
	Passing to global sections in the last square and pre-composing with the adjunction unit, we see that $\Gamma(h)$ induces an morphism \[(\on{Tot}\check{C}(U,A),\check{\theta})\to (\Gamma(\mc{T},I),\theta)\] in $\mc{C}(p)$. Since (\ref{cech-map}) is induced by $\Gamma(h)$, it is compatible with Steenrod operations, as desired.	
\end{proof}

\begin{rmk}
	(i) The morphism (\ref{cech-map}) is defined in \cite[Tag 08BN]{stacks-project}, at least when $\mc{T}$ is the ringed topos of a ringed space $(X,\mc{O}_X)$. The construction easily adapts to the case of arbitrary topoi; this is how we constructed (\ref{cech-map}). 
	
	(ii) The main ingredient in our proof of \Cref{cech-to-derived} is \Cref{comes-from-sheaf}, which shows that (\ref{cech-map}) comes from a map of sheaves. This is crucial, because in order to compare the Steenrod operations of May with those of Drury in \Cref{cech-to-derived} one cannot pass to global sections too soon. More precisely, we cannot prove directly that the outer square of global section is $\pi$-homotopy commutative, without first showing that the outer square of sheaves is $\pi$-homotopy commutative. This is because $\Gamma(\mc{T},\nu):\Gamma(\mc{T},A)\to \Gamma(\mc{T},I)$ is not necessarily a quasi-isomorphism.
\end{rmk}

	\section{De Rham cohomology of stacks and approximation arguments}\label{approx}
	Let $p$ be a prime number, let $k$ be a field of characteristic $p$, and let $X$ be a smooth algebraic stack over $k$. We denote by $\Omega_{X/k}$ the de Rham complex of $X$, viewed as a complex of big \'etale sheaves on $X$. We write $H^*_{\on{dR}}(X/k)$ for the de Rham cohomology of $X$, that is, the hypercohomology of $\Omega_{X/k}$. We consider the following property of $X$.
	
	\begin{property}\label{property}
		For every $d\geq 0$, there exist a smooth $k$-scheme of finite type $Z_d$  and a morphism $Z_d\to X$ such that the induced map $H^*_{\on{dR}}(X/k)\to H^*_{\on{dR}}(Z_d/k)$ is injective in degrees $\leq d$. 	
	\end{property}
	
	We denote by $D(k)$ the derived category of $k$-vector spaces, and for all $h\in \Z$ we let $D(k)^{\geq h}$ be the subcategory of $D(k)$ consisting of complexes with cohomology  equal to zero in all degrees $<h$. Let $f:X\to Y$ be a morphism of smooth $k$-stacks, and let $d\geq 0$ be an integer. Following \cite[Definition 5.1]{antieau2019counterexamples}, we say that $f$ is a $d$-Hodge equivalence if for all $j\geq 0$ the cone of $Rf^*:R\Gamma(Y,\Omega^j_{Y/k})\to R\Gamma(X,\Omega^j_{X/k})$ belongs to $D(k)^{\geq d-j}$. If $f$ is a $d$-Hodge equivalence, then by \cite[Remark 5.1]{antieau2019counterexamples} the induced map $f^*:H^i_{\on{dR}}(Y/k)\to H^i_{\on{dR}}(X/k)$ is an isomorphism for all $i<d$.

	\begin{lemma}\label{local-coh}
		Let $X$ be a smooth $k$-scheme of finite type, let $Z\c X$ be a closed subscheme of codimension $d+1$, and let $U:=X\setminus Z$. Then the inclusion $U\hookrightarrow X$ is a $d$-Hodge equivalence.
	\end{lemma}
	
	\begin{proof}
	If $i\geq 0$ is an integer and $F$ is a sheaf on $X$, we denote by $H^i(U,F)$ the value at $F$ of the $i$-th derived functor of $\Gamma(U,-)$, viewed as a functor from abelian sheaves on $X$ to abelian groups. Using \cite[Tag 01E1]{stacks-project} and the identification $\Omega_{U/k}^j=\Omega_{X/k}^j|_U$, we see that \[H^*(U,\Omega_{X/k}^j)=H^*(U,\Omega_{X/k}^j|_U)=H^*(U,\Omega_{U/k}^j).\] Thus it suffices to show that the natural map $H^i(X,\Omega_{X/k}^j)\to H^i(U,\Omega_{X/k}^j)$ is an isomorphism for all $0\leq i\leq d$. 
		
		Since $X$ is a smooth $k$-scheme, $\Omega^j_{X/k}$ is a locally free sheaf on $X$. Recall that, if $R$ is a local noetherian ring and $M$ and $N$ are finitely generated $R$-modules, then $\on{depth}(M\oplus N)=\min\set{\on{depth}M,\on{depth}N}$. (To see this, use the Ext characterization of depth.) Thus, for all $x\in X$ we have \[\on{depth}\Omega_{X/k,x}^j=\on{depth}\mc{O}_{X,x}=\dim \mc{O}_{X,x}=\on{codim}_X\cl{\set{x}}.\] Here we are computing depths of $\mc{O}_{X,x}$-modules. The last equality is the fact that a regular local ring is Cohen-Macaulay. Therefore, for all $z\in Z$ we have \[\on{depth}\Omega_{X/k,z}^j\geq \on{codim}_XZ= d+1.\]
		By \cite[Expos\'e III, Proposition 3.3]{sga2}, this implies that $H^i(X,\Omega_{X/k}^j)\to H^i(U,\Omega_{X/k}^j)$ is an isomorphism for all $0\leq i\leq d$, as desired.
	\end{proof}

	\begin{prop}\label{approximate}
		Let $Y$ be a smooth quasi-projective $k$-scheme, let $G$ is a linear algebraic $k$-group, and let $X:=[Y/G]$. Then $X$ satisfies \Cref{property}.
	\end{prop}
	
	\begin{proof}
		Let $d\geq 0$ be an integer. There exist a representation $V$ of $G$, a closed $G$-invariant subscheme $Z\c V$ of codimension $\geq d+2$, such that the complement $U:=V\setminus Z$ is the total space of a $G$-torsor $U\to U/G$, where $U/G$ is a smooth quasi-projective $k$-scheme. By \Cref{local-coh}, the inclusion $U\hookrightarrow V$ is a $(d+1)$-Hodge equivalence. By \cite[Proposition 5.10(1)]{antieau2019counterexamples}, the inclusion $U\times_kY\hookrightarrow V\times_kY$ is also a $(d+1)$-Hodge equivalence, which is $G$-equivariant if we let $G$ act diagonally. By \cite[Proposition 5.10(2)]{antieau2019counterexamples}, the induced open embedding $(U\times_kY)/G\hookrightarrow [(V\times_kY)/G]$ is a $(d+1)$-Hodge equivalence. By \cite[Remark 5.1]{antieau2019counterexamples}, this implies that for all $i\leq d$ the natural maps
		\[H^i_{\on{dR}}([(V\times_kY)/G]/k)\to H^i_{\on{dR}}(((U\times_kY)/G)/k)\]
		are isomorphisms.
		
		Using the zero section of the vector bundle $[(V\times_kY)/G]\to [Y/G]$, we see that $H^*_{\on{dR}}(X/k)$ is a direct summand of $H^*_{\on{dR}}([(V\times_kY)/G]/k)$. We conclude that for all $i\leq d$ the pullback homomorphisms \[H^i_{\on{dR}}(X/k)\to H^i_{\on{dR}}(((U\times_kY)/G)/k)\] are injective. As $Y$ is quasi-projective, $(U\times_kY)/\on{GL}_n$ is a scheme, hence \[Z_d:=(U\times_kY)/G\to X\] satisfies \Cref{property}.
	\end{proof}

\begin{rmk}
	(i) Let $G\hookrightarrow \on{GL}_n$ be a faithful representation of $G$ over $k$, for some $n\geq 2$. Then $X=[Y/G]\cong [(Y\times^G\on{GL}_n)/\on{GL}_n]$. Here $Y\times^G\on{GL}_n$ is the fppf-quotient of $Y\times_k \on{GL}_n$ by the diagonal action of $G$. Since $Y$ is smooth and quasi-projective, the fppf-quotient is represented by a smooth quasi-projective $k$-scheme. Thus, in the course of proving \Cref{approximate}, we could have assumed that $G=\on{GL}_n$ for some $n\geq 1$. 
	
	When $G=\on{GL}_n$ for some $n\geq 1$, we can be more explicit about the representations that we use. Namely, let $r\geq 0$ be an integer, let $V:=M_{n,n+r}$, the $k$-vector space of $n\times (n+r)$-matrices, on which $\on{GL}_n$ acts by multiplication on the left, let $Z$ be the locus of matrices of rank $<n$, and $U:=V\setminus Z$. If $r$ is sufficiently large, $Z$ has codimension $\geq d+1$ in $V$. We have a $\on{GL}_n$-torsor $U\to U/\on{GL}_n$, where $U/\on{GL}_n=\on{Gr}(n,n+r)$ is a Grassmannian.
	
	(ii) In the setting of \Cref{local-coh}, assume that $k$ is of characteristic zero. Then $H^i_{\on{dR}}(X/k)\to H^i_{\on{dR}}(U/k)$ is injective for all $0\leq i\leq 2d+1$. The proof goes by induction, cutting $X$ by a smooth divisor $Y$ and using the logarithmic de Rham complex $\Omega_{X/k}(\on{log}Y)$; see \cite[Lemma p. 11]{dejong2007notes}.
	
	(iii) When $X=BG$ for a finite or reductive $k$-group $G$, \Cref{property} has been established by B. Antieau, B. Bhatt, and A. Mathew in \cite[Theorem 1.2]{antieau2019counterexamples} by a completely different argument. The $Z_d$ exhibited by them are smooth projective $k$-varieties.
\end{rmk}

	\section{Steenrod operations on de Rham cohomology}\label{derham} 
	Let $p$ be a prime number, let $k$ be a field of characteristic $p$, and let $X$ be a smooth algebraic stack of finite type over $k$. We apply the construction of \Cref{definesteenrod} to the case where $\mc{T}$ is the big \'etale topos of $X$ and $A=\Omega_{X/k}$ is the de Rham complex, viewed as a commutative differential graded $\F_p$-algebra on $\mc{T}$. We thus obtain Steenrod operations on $H^*_{\on{dR}}(X/k)$. 
	
	The purpose of this section is to show, under additional assumptions on $X$, that negative Steenrod operations on $H^*_{\on{dR}}(X/k)$ are zero, and to compute $\on{Sq}^0$ and $\on{P}^0$.
	
	\subsection{Compatibility with cohomology on the crystalline site}
	Let $X$ be a smooth scheme over $k$. Let $X_{\text{Zar}}$, $X_{\text{ZAR}}$, $X_{\text{\'et}}$, $X_{\text{\'ET}}$. denote the small and big Zariski site of $X$, and the small and big \'etale site of $X$, respectively. We may view $\Omega_{X/k}$ as a complex of sheaves on each of these four sites. We have an obvious commutative diagram of morphisms of sites:
	\[
	\begin{tikzcd}
	X_{\text{Zar}} \arrow[r] \arrow[d] & X_{\text{\'et}} \arrow[d] &\\ 
	X_{\text{ZAR}} \arrow[r] &  X_{\text{\'ET}}.
	\end{tikzcd}
	\]
	Since $X$ is a scheme, $\Omega_{X/k}$ is a complex of coherent sheaves on $X$, hence each map in the previous diagram induces an isomorphism in the hypercohomology of $\Omega_{X/k}$. Applying \Cref{functorial} to the morphisms of topoi associated to each of the arrows appearing, we see that the Steenrod operations on $H^*_{\on{dR}}(X/k)$ do not depend on the choice of site.
	
	We write $(X/k)_{\on{cris}}$ for the crystalline site of $X$ over $k$, where we regard $k$ as a divided power ring, with the unique divided power structure with respect to the ideal $(0)$. Let $(u_*,u^{-1}):(X/k)_{\on{cris}}\to X_{\on{Zar}}$ be the morphism of topoi defined in \cite[Proposition 5.18]{berthelot1978notes}.\footnote{In \cite{berthelot1978notes}, $u^{-1}$ is denoted by $u^*$.} By  \cite[5.19]{berthelot1978notes}, the unit $\on{id}\to u_*u^{-1}$ is an isomorphism of functors, and in particular we have a canonical isomorphism
	\begin{equation}\label{unit}
	\Omega_{X/k}\xrightarrow{\sim}u_*u^{-1}\Omega_{X/k}.
	\end{equation}
	
	By the Poincar\'e Lemma in crystalline cohomology \cite[Theorem 6.12]{berthelot1978notes}, there is a canonical quasi-isomorphism $\mc{O}_{X/k}\to L(\Omega_{X/k})$ of complexes of abelian sheaves on $(X/k)_{\on{cris}}$, where $\mc{O}_{X/k}$ is viewed as a complex concentrated in degree zero. Here $L$ is the linearization functor of \cite[Construction 6.9 and below]{berthelot1978notes}. Applying \cite[Proposition 6.10]{berthelot1978notes} to $S=\Spec k$ and $Y=X$, we see that there is a natural isomorphism $L(\Omega_{X/k})\xrightarrow{\sim}u^{-1}\Omega_{X/k}$. Composing these two maps, we obtain a canonical quasi-isomorphism
	\[
	\psi:\mc{O}_{X/k}\xrightarrow{\sim} u^{-1}\Omega_{X/k}
	\]
	on $(X/k)_{\on{cris}}$. Moreover, $\psi$ is a homomorphism of sheaves of differential $k$-algebras.
	
	We have $\Gamma((X/k)_{\on{cris}},-)=\Gamma(X_{\on{Zar}},-)\circ u_*$, and so \begin{equation}\label{derivedfunctors}
	R\Gamma((X/k)_{\on{cris}},-)=R\Gamma(X_{\on{Zar}},-)\circ Ru_*.
	\end{equation} 
	Since $u^{-1}\Omega_{X/k}$ is acyclic with respect to $u_*$, the natural map 
	\begin{equation}\label{Ruu}u_*u^{-1}\Omega_{X/k}\to Ru_*u^{-1}\Omega_{X/k}\end{equation} 
	is an isomorphism in the derived category; see \cite[Corollary 5.27]{berthelot1978notes} or the proof of \cite[Theorem 7.1]{berthelot1978notes}. The composition of $Ru_*(\psi)$, the inverse of (\ref{Ruu}) and the inverse of (\ref{unit}) yields a canonical isomorphism 
	\begin{equation}\label{poincare} Ru_{*}(\mc{O}_{X/k})\xrightarrow{\sim} \Omega_{X/k}\end{equation} 
	in the derived category of sheaves of abelian groups on $X_{\on{Zar}}$. It may be helpful for the reader to note that this is exactly the isomorphism of \cite[(7.1.2)]{berthelot1978notes}. 
	Using (\ref{derivedfunctors}), we obtain a canonical isomorphism
	\begin{equation}\label{poincare-global}H^*((X/k)_{\on{cris}},\mc{O}_{X/k})\xrightarrow{\sim} \mathbb{H}^*(X_{\on{Zar}},\Omega_{X/k})=H^*_{\on{dR}}(X/k).\end{equation}
	Since $\mc{O}_{X/k}$ is a sheaf of $k$-algebras on $(X/k)_{\on{cris}}$, on the left side of (\ref{poincare-global}) we also have Steenrod operations.
	
	\begin{prop}\label{cris-compatible-scheme}
		Let $X$ be a smooth scheme over $k$. For every $i\geq 0$, the isomorphism (\ref{poincare-global}) is compatible with Steenrod operations.
	\end{prop}
	
	\begin{proof}
	We apply \Cref{functorial} to $\mc{T}=\mc{T}'=(X/k)_{\on{cris}}$, $(f_*,f^{-1})=(\on{id},\on{id})$, $A=f^{-1}A=\mc{O}_{X/k}$, $A'=u^{-1}\Omega_{X/k}$, and $\psi$ as the map $f^{-1}A'\to A$. We obtain that the isomorphism \[\psi^*:H^*((X/k)_{\on{cris}},\mc{O}_{X/k})\xrightarrow{\sim} \mathbb{H}^*((X/k)_{\on{cris}},u^{-1}\Omega_{X/k})\] is compatible with Steenrod operations.
	
	We now apply \Cref{functorial} to $\mc{T}'=(X/k)_{\on{cris}}$, $\mc{T}=X_{\on{Zar}}$, $(f_*,f^{-1})=(u_*,u^{-1})$, $A=\Omega_{X/k}$, $A'=u^{-1}\Omega_{X/k}$, and the identity as homomorphism $f^{-1}A\to A'$. We obtain that the induced homomorphism
	\[H^*_{\on{dR}}(X/k)\to \mathbb{H}^*((X/k)_{\on{cris}},u^{-1}\Omega_{X/k})\] is compatible with Steenrod operations. 
	
	By \Cref{commutes-ruu}, we see that this is in fact the isomorphism induced by the inverse of (\ref{Ruu}). Thus (\ref{poincare-global}) is compatible with Steenrod operations, as desired.
	\end{proof}

	\begin{rmk}\label{frobenius}
	Let $\on{Fr}: \mc{O}_{X/k}\to \mc{O}_{X/k}$ be the Frobenius endomorphism, and let $\Phi: \Omega_{X/k}\to \Omega_{X/k}$ be the endomorphism given by the Frobenius $\mc{O}_X\to \mc{O}_X$ in degree zero, and zero in all other degrees. Note that $\Phi$ is an endomorphism of complexes because $d(f^p)=pf^{p-1}d(f)=0$ for every section $f$ of $\mc{O}_X$. We have a commutative square
	\[
	\begin{tikzcd}
	\mc{O}_{X/k} \arrow[r,"\psi"] \arrow[d,"\on{Fr}"] & u^{-1}\Omega_{X/k} \arrow[d,"u^{-1}\Phi"] \\
	\mc{O}_{X/k} \arrow[r,"\psi"] & u^{-1}\Omega_{X/k}.   
	\end{tikzcd}
	\]
	It follows that we have a commutative diagram
	\[
	\begin{tikzcd}
	H^*((X/k)_{\on{cris}},\mc{O}_{X/k}) \arrow[r,"\psi^*"] \arrow[d,"\on{Fr}^*"] & \mathbb{H}^*((X/k)_{\on{cris}}, u^{-1}\Omega_{X/k}) \arrow[d,"(u^{-1}\Phi)^*"] & \arrow[l]  H^*_{\on{dR}}(X/k) \arrow[d,"\Phi^*"]  \\
	H^*((X/k)_{\on{cris}},\mc{O}_{X/k}) \arrow[r,"\psi^*"] & \mathbb{H}^*((X/k)_{\on{cris}}, u^{-1}\Omega_{X/k}) &\arrow[l]  H^*_{\on{dR}}(X/k),  
	\end{tikzcd}
	\]	
	where the horizontal arrows on the right come from \Cref{functorial} as in the proof of \Cref{cris-compatible-scheme}. As we observed in the course of proving \Cref{cris-compatible-scheme}, the horizontal maps on the right are isomorphisms, and the composition of their inverse with $\psi^*$ are (\ref{poincare-global}). It follows that (\ref{poincare-global}) transports the Frobenius endomorphism to $\Phi^*$. We will use this remark during the proof of \Cref{steenrod-0} below.
	\end{rmk}
	
	\subsection{Vanishing of negative operations}
	Let $X$ be a smooth algebraic stack over $k$. 
	
	\begin{prop}\label{cris-compatible}
		If $X$ satisfies \Cref{property}, the negative Steenrod operations on $H^*_{\on{dR}}(X/k)$ are zero.
	\end{prop}
	
	\begin{proof}
		By \Cref{functorial} and \Cref{approximate}, we may assume that $X$ is a smooth scheme of finite type over $k$.	By \Cref{cris-compatible-scheme}, it suffices to show that the negative Steenrod operations on $H^*((X/k)_{\on{cris}},\mc{O}_{X/k})$ are zero. Since $\mc{O}_{X/k}$ is concentrated in degree $0$, this follows from \Cref{steenrod<0}.
	\end{proof}

	\subsection{Determination of $\on{Sq}^0$ and $\on{P}^0$}
	
	Let $\mc{A}$ be an abelian category with sufficiently many injectives, let $\mc{B}$ be an abelian category, and let $F:\mc{A}\to \mc{B}$ be a left exact additive functor. Let $(A,d)$ be a cochain complex with $A^i\in \mc{A}$ for all $i$ and $A^i=0$ for $i<0$. We have the two hypercohomology spectral sequences
	\[{'E}_1^{rs}:= R^sF(A^r)\Rightarrow H^{r+s}(RF(A))\]
	and 
	\[{''E}_2^{rs}:= R^rF(H^s(A))\Rightarrow H^{r+s}(RF(A)).\] Let $A\to I^{**}$ be a Cartan-Eilenberg resolution of $A$; see \cite[Tag 015H]{stacks-project}. The two spectral sequences are obtained from the two spectral sequences associated to the double complex $F(I^{**})$; see \cite[Tag 015J]{stacks-project}. Associated to these spectral sequences, we have edge homomorphisms \['e^i:H^i(RF(A))\to R^iF(A^0),\qquad ''e^i: R^i(H^0(A))\to H^i(RF(A)).\]
		
	Now let $\phi\in \on{Hom}(A^0,A^0)$, and assume that $d^1\phi=0$. We also denote by $\phi$ the induced map $A^0\to \on{Ker}(d^1)=H^0(A)$. We have a cochain map $\Phi:A\to A$, given by $\Phi^0:=\phi$ and $\Phi^i:=0$ for all $i>0$. 
	
	\begin{lemma}\label{factors}
	The diagram
	\[
	\begin{tikzcd}
	H^*(RF(A))\arrow[r, "\Phi^*"] \arrow[d,"{'e}"]  & H^*(RF(A))  \\ 
	R^*F(A^0) \arrow[r, "\phi^*"]  & \arrow[u,"{''e}"] R^*F(H^0(A))
	\end{tikzcd}
	\]
	is commutative.
	\end{lemma}
	
	\begin{proof}
	By assumption, the map $\Phi:A\to A$ factors as \[A\to A^0[0]\xrightarrow{\phi} H^0(A)[0]\to A.\] We deduce that $\Phi^*$ factors as
	\[H^*(RA(F))\xrightarrow{\psi'} H^*(RF(A^0[0]))\xrightarrow{\phi^*}H^*(RF(H^0(A)[0]))\xrightarrow{\psi^{''}}H^*(RA(F)).\]
	Here $\psi'$ and $\psi''$ come from the functoriality of the $E_{\infty}$-pages of the first and second spectral sequence, respectively. Consider now the following commutative diagram:
		\[
	\begin{tikzcd}
	H^*(RA(F))\arrow[r,"\psi'"] \arrow[d,equal] & H^*(RF(A^0[0])) \arrow[r,"\phi^*"] \arrow[d,"\wr"]  &  H^*(RF(H^0(A)[0])) \arrow[r,"\psi''"] \arrow[d,"\wr"]  & H^*(RA(F)) \arrow[d,equal]  \\ 
	H^*(RA(F)) \arrow[r,"'e"]  & \arrow[r,"\phi^*"] R^*F(A^0) & R^*F(H^0(A)) \arrow[r,"''e"] & H^*(RF(A)). 
	\end{tikzcd}
	\]
	The commutativity of the square on the left (resp. right) follows from the naturality of the edge maps in the first (resp. second) hypercohomology spectral sequence. The square in the middle is commutative, since the $E_2$-pages of the first and second spectral sequences for the hypercohomology of a complex concentrated in degree $0$ coincide.
	\end{proof}
	
	We immediately obtain the following result.
	
	\begin{prop}\label{steenrod-0}
	Assume that $X$ satisfies \Cref{property}, and let $\on{Sq}^0$ (if $p=2$) and $\on{P}^0$ (if $p>2$) be the zeroth Steenrod operation on $H^*_{\on{dR}}(X/k)$.
	
	(a) Then $\on{P}^0$ and $\on{Sq}^0$ factor as
	\[H^*_{\on{dR}}(X/k)\to H^*(X,\mc{O}_X)\to H^*(X,\mc{O}_X)\to H^*_{\on{dR}}(X/k),\] where the first map is an edge homomorphism in the Hodge spectral sequence, the second map is induced by the Frobenius endomorphism of $\mc{O}_X$ and the third map is an edge homomorphism in the conjugate spectral sequence.
	
	(b) The composition \[H^0(X,\Omega_{X/k}^1)\hookrightarrow H^1_{\on{dR}}(X/k)\xrightarrow{\on{P}^0} H^1_{\on{dR}}(X/k)\] is equal to zero.
	\end{prop}
	
	\begin{proof}
	By \Cref{property} and \Cref{functorial}, we may assume that $X$ is a smooth scheme of finite type over $k$. By \Cref{steenrod-0}, $\on{Sq}^0$ and $\on{P}^0$ are induced by the Frobenius endomorphism of $\mc{O}_{X/k}$. We deduce from \Cref{frobenius} that $\on{Sq}^0$ and $\on{P}^0$ are induced by the homomorphism $\Phi:\Omega_{X/k}\to \Omega_{X/k}$ given by the Frobenius $\phi:\mc{O}_X\to \mc{O}_X$ in degree zero, and zero everywhere else.
		
	(a) The Hodge and conjugate spectral sequences for $X$ are a special case of the first and second hypercohomology spectral  sequence, letting $\mc{A}$ be the category of sheaves of $k$-vector spaces over $X$, $A$ be the de Rham complex of $X$, and $\Phi$ and $\phi$ as in the previous paragraph. Now (a) follows from \Cref{factors}.	
		
	(b) This follows from (a) and the exact sequence of low degree terms in the Hodge spectral sequence:
	\[0\to H^0(X,\Omega_{X/k}^1)\to H^1_{\on{dR}}(X/k)\to H^1(X,\mc{O}_X).\qedhere\]
	\end{proof}

	\begin{rmk}\label{cris-frobenius}
	The Steenrod operations $\on{P}^0$ and $\on{Sq}^0$ are not equal to the identity; see \Cref{steenrod<0}(b) or \Cref{steenrod-0}(b). From \Cref{mainthm}(a) (to be proved in the next section) we also see that they are not identically zero. As another example, let $E$ be an elliptic curve over $\F_p$. Then the Steenrod operations in degree $0$ are trivial on $H^*_{\on{dR}}(E/\F_p)$ if and only if $E$ is ordinary. More generally, the behavior of $\on{Sq}^0$ and $\on{P}^0$ is related to the Frobenius and Hodge filtration on crystalline cohomology; see e.g. \cite[Chapter 8]{berthelot1978notes}. We will not make use of this remark in the sequel.
	\end{rmk}

	\begin{rmk}
	(i) It would be interesting to know whether \Cref{property} holds for an arbitrary smooth stack of finite type over $k$. By \Cref{steenrod-0}, this would imply the validity \Cref{main-steenrod}(iv) holds for all such stacks. 
	
	(ii) One could remove the requirement that $Y$ be quasi-projective by showing that the crystalline Poincar\'e Lemma for algebraic spaces of M. Olsson \cite[Corollary 2.5.4]{olsson2007crystalline} is compatible with Steenrod operations. We chose not to do so in this paper, in order to remain within classical crystalline cohomology. Note that the crystalline Poincar\'e Lemma for algebraic $k$-stacks (even Deligne-Mumford $k$-stacks) is not known. For example, \cite{olsson2007crystalline} only addresses the situation of a representable morphism $X\to S$, where $X$ and $S$ are Deligne-Mumford stacks. If one had a Poincar\'e Lemma for algebraic $k$-stacks, one could probably prove \Cref{main-steenrod}(iv) more generally and without recourse to \Cref{property}.
	\end{rmk}

	\begin{prop}
	Let $X$ be a smooth affine $k$-scheme. Then the $p$-th power Steenrod operations on $H^*_{\on{dR}}(X/k)$ are trivial.
	\end{prop}

	\begin{proof}
	We apply \Cref{cech-to-derived} to $\mc{S}$ the big \'etale (or Zariski) site of $X$, $U$ the Cech nerve of the identity $X\to X$, and $A=\Omega_{X/k}$. We deduce that the Steenrod operations on $H^*_{\on{dR}}(X/k)$ are trivial if and only if the Steenrod operations on $H^*(\on{Tot}\check{C}(U,\Omega_{X/k}))$ are trivial. It is easy to see that $\Gamma(X,\Omega_{X/k})$ is a summand of $H^*(\on{Tot}\check{C}(U,\Omega_{X/k}))$, and that the projection $\on{Tot}\check{C}(U,\Omega_{X/k})\to \Gamma(X,\Omega_{X/k})$ is a quasi-isomorphism. The summand $\Gamma(X,\Omega_{X/k})$ arises as the totalization of the cosimplicial submodule of $\check{C}_0(U,\Omega_{X/k})$ given by $\Gamma(X,\Omega_{X/k})$ in level zero, and zero in higher levels. Thus there is a commutative square
	\[
	\begin{tikzcd}
	\mc{Z}(p)\otimes \on{Tot}\check{C}(U,\Omega_{X/k})^{\otimes p}\arrow[r] \arrow[d] & \on{Tot}\check{C}(U,\Omega_{X/k}) \arrow[d]\\
	\mc{Z}(p)\otimes \Gamma(X, \Omega_{X/k})^{\otimes p} \arrow[r] & \Gamma(X, \Omega_{X/k}),
	\end{tikzcd}
	\]
	where the vertical arrows are the natural projections, the top horizontal arrow is (\ref{zp}), and the bottom arrow is a direct summand of the top arrow. As $\Gamma(X,\Omega_{X/k})$ is the totalization of the degree zero component of $\check{C}_0(U,\Omega_{X/k})$, it follows from the definitions that the bottom arrow is the tensor product of the augmentation of $\mc{Z}(p)$ and the $p$-fold multiplication of $\Omega_{X/k}$. Pre-composing with the quasi-isomorphism $W\to \mc{Z}(p)$ of (\ref{w-e-z}), we obtain a commutative square
	\[
	\begin{tikzcd}
	W\otimes \on{Tot}\check{C}(U,\Omega_{X/k})^{\otimes p}\arrow[r,"\check{\theta}"] \arrow[d] & \on{Tot}\check{C}(U,\Omega_{X/k}) \arrow[d]\\
	W\otimes \Gamma(X, \Omega_{X/k})^{\otimes p} \arrow[r,"\epsilon\otimes m_p"] & \Gamma(X, \Omega_{X/k}).
	\end{tikzcd}
	\]
	Here $\check{\theta}$ is the map of (\ref{check-global}), $\epsilon$ is the augmentation of $W$, and $m_p$ is the $p$-fold multiplication map. As the projection $\on{Tot}\check{C}(U,\Omega_{X/k})\to \Gamma(X,\Omega_{X/k})$ is a quasi-isomorphism, the conclusion will follow if we can show that the $p$-th power Steenrod operations on $H^*(\Gamma(X, \Omega_{X/k}))$ associated to the object $(\Gamma(X, \Omega_{X/k}),\epsilon\otimes m_p)$ of $\mc{C}(p)$ are trivial. This is true by \Cref{associative}. 
	\end{proof}

	\subsection{Bockstein homomorphism}\label{bock-cech}
	
	\begin{lemma}\label{smooth-etale}
		Let $k$ be a field, let $X$ be an algebraic stack over $k$, and denote by $X_{\on{\acute{E}T}}$ and $X_{\on{SM}}$ the big \'etale and big smooth site of $X$, respectively.
		
		(a) Let $F$ be an \'etale sheaf of abelian groups on $X$. Then $F$ is a smooth sheaf on $X$, and the canonical homomorphism
		\[H^*(X_{\on{\acute{E}T}},F)\to H^*(X_{\on{SM}},F)\] 
		is an isomorphism.
		
		(b) Assume that $\on{char}k=p>0$, and that $X$ is smooth over $k$. The canonical homomorphism \[H^*_{\on{dR}}(X/k)\to \mathbb{H}^*(X_{\on{SM}},\Omega_{X/k})\] is an isomorphism, compatible with Steenrod operations.
	\end{lemma}

	\begin{proof}
		(a) The fact that $F$ is a smooth sheaf follows from the fact that a covering in the smooth topology can always be refined by a covering in the \'etale topology; see \cite[Tags 005V, 00VX]{stacks-project}.
		
		(b) By (a), $\Omega_{X/k}$ is a complex of smooth sheaves. We have a homomorphism between the \'etale hypercohomology spectral sequence for $\Omega_{X/k}$ to the smooth hypercohomology spectral sequence for $\Omega_{X/k}$. Applying (a) to $F=\Omega_{{X}/k}^j$ for every $j\geq 0$, we see that this homomorphism is an isomorphism between the $E_2$-pages, hence it induces an isomorphism of the abutments. The compatiblity with Steenrod operations is a special case of \Cref{functorial}. 
	\end{proof}
	
	Let $(K,\theta)\in \mc{C}(p)$. Recall that one may define a Bockstein homomorphism $\beta$ on $H^*(K)$, assuming that $(K,\theta)$ is reduced, that is, that there exists $(\tilde{K},\tilde{\theta})\in \mc{C}(p,\Z/p^2\Z)$ whose reduction modulo $p$ is isomorphic to $(K,\theta)$ and such that $\tilde{K}$ is flat over $\Z/p^2\Z$.
	
	Let $X$ be a smooth algebraic stack of finite type over a commutative ring $R$, and assume that $X$ has affine diagonal over $R$. By \cite[Tag 04YA]{stacks-project}, there exists a smooth surjective morphism from a smooth affine $R$-scheme of finite type to $X$. Let $U$ be the \v{C}ech nerve of a covering of $X$ by a smooth affine $R$-scheme. Then $\Gamma(U_n,\Omega_{X/R}^j)=\Gamma(U_n,\Omega_{U_n/R}^j)$ is a flat $R$-module for every $n,j\geq 0$, and so $\on{Tot}\check{C}(U,\Omega_{X/R})$ is a complex of flat $R$-modules. 
	
	Let $R\to S$ be a ring homomorphism. By the theorem on cohomology and affine base change for quasi-coherent sheaves \cite[Tag 02KG]{stacks-project}, for all $n,j\geq 0$ the canonical map $\Omega^j_{U_n/R}\otimes_RS\to \Omega^j_{(U_n)_S/S}$ is an isomorphism of quasi-coherent $\mc{O}_{X_S}$-modules. Moreover, the canonical maps are compatible with the de Rham differentials. 
	We thus obtain an isomorphism of complexes of  $S$-modules
	\begin{equation}\label{derham-iso}\on{Tot}\check{C}(U,\Omega_{X/R})\otimes_RS\xrightarrow{\sim} \on{Tot}\check{C}(U_S,\Omega_{X_S/S}).\end{equation}
	
		\begin{prop}\label{bockstein}
		Let $k$ be a perfect field of characteristic $p$, and let $X$ be a smooth algebraic stack over $k$. Assume that there exists a smooth algebraic stack $\tilde{X}$ of finite type and with affine diagonal over $W_2(k)$, such that $X=\tilde{X}\times_{W_2(k)}k$. Then there exists a group homomorphism \[\beta: H^*_{\on{dR}}(X/k)\to H^{*+1}_{\on{dR}}(X/k),\] satisfying properties (i) and (ii) of \Cref{bockstein-properties}.
	\end{prop}

\begin{proof}
	Let $\tilde{U}$ be the \v{C}ech nerve associated to a smooth surjective morphism from a smooth affine $W_2(k)$-scheme to $\tilde{X}$, and define $U:=\tilde{U}\times_{W_2(k)}k$. For every $n\geq 0$, $\tilde{U}_n$ is a smooth and affine $W_2(k)$-scheme. 
	For all $n,j\geq 0$, letting $\eta_n:U_n\to X$ denote the natural projection map, we have an isomorphism $\Omega_{U_n/k}^j\cong \eta_n^{-1}\Omega^j_{X/k}$ of big \'etale sheaves on $U_n$. Since ${U}_n$ is affine and $\Omega_{U_n/k}^j$ is a coherent sheaf on ${U}_n$ for every $n,j\geq 0$, by Serre's Vanishing Theorem we have $H^q(U_n,\Omega_{U_n/k}^j)=0$ for all $q>0$ and $n,j\geq 0$. 
	
	Let $\mc{S}$ be the site whose objects are algebraic stacks over $X$, whose morphisms are $1$-morphisms of algebraic stacks over $X$, and whose covers are families of jointly surjective smooth morphisms, and let $\mc{T}$ be the topos associated to $\mc{S}$. With the notation of \Cref{smooth-etale}, we have morphism of sites \[X_{\on{\acute{E}T}}\to X_{\on{SM}}\to \mc{S},\]
	
	Since every algebraic stack has a smooth cover by schemes, by Verdier's Comparison Theorem \cite[III, Th\'eor\`eme 4.1]{sga4I} the morphism on the right induces an equivalence of topoi. In particular, we have an induced isomorphism \[\mathbb{H}^*(X_{\on{SM}},\Omega_{{X}/k})\xrightarrow{\sim}\mathbb{H}^*(\mc{S},\Omega_{X/k})=\mathbb{H}^*(\mc{T},\Omega_{X/k}),\]  where we also denote by $\Omega_{X/k}$ the sheaf on $\mc{S}$ induced by the de Rham complex on $X_{\on{SM}}$. This isomorphism is compatible with Steenrod operations by \Cref{functorial}. Combining this with \Cref{smooth-etale}, we get a canonical isomorphism \[H^*_{\on{dR}}(X/k)\xrightarrow{\sim}\mathbb{H}^*(\mc{T},\Omega_{X/k}),\] which is compatible with Steenrod operations. 
	
	Note that $X$ is a terminal object in $\mc{S}$, and that a smooth surjective morphism from a scheme to $X$ is a cover in $\mc{S}$. Moreover, by \Cref{smooth-etale} and \cite[Tags 06W0, 0DGB, 03P2]{stacks-project}, for all $j,n\geq 0$ and $i\geq 1$ we have
	\begin{align*}H^i(\mc{S}/U_n,\eta_n^{-1}\Omega_{X/k}^{j})&=H^i((\on{Sch}/U_n)_{\on{SM}},\Omega_{U_n/k}^j)\\ &=H^i((\on{Sch}/U_n)_{\on{\acute{E}T}},\Omega_{U_n/k}^j)\\ 
	&=H^i((U_n)_{\on{\acute{e}t}},\Omega_{U_n/k}^j)\\
	&=H^i(U_n,\Omega_{U_n/k}^j)=0.\end{align*}
	In the last step we have used Serre's Vanishing Theorem for the cohomology of quasi-coherent sheaves on affine schemes. Thus, we may apply \Cref{cech-to-derived} to $\mc{T}$, $\Omega_{X/k}$ and $U$. 
	We deduce the existence of an isomorphism 
	\begin{equation}\label{iso-to-reduced}
	H^*(\on{Tot}\check{C}({U},\Omega_{{X}/k}))\xrightarrow{\sim}H^*_{\on{dR}}(X/k)
	\end{equation}
	compatible with Steenrod operations. 
	
	 The map (\ref{check-global}) is compatible with extension of scalars, thus (\ref{derham-iso}) induces an isomorphism
	\[(\on{Tot}\check{C}(\tilde{U},\Omega_{\tilde{X}/W_2(k)}),\check{\theta})\otimes_{\Z/p^2\Z}\F_p\cong (\on{Tot}\check{C}({U},\Omega_{{X}/k}),\check{\theta})\] in $\mc{C}(p)$. For every $n\geq 0$, since $\tilde{U}_n$ is smooth over $W_2(k)$, by \cite[Tag 02G1]{stacks-project} the $W_2(k)$-module $\Omega_{\tilde{U}_n/W_2(k)}$ is finite locally free, and in particular it is flat. Recall that $W_2(k)$ is flat over $\Z/p^2\Z$; in fact, it is the unique flat lifting of $\F_p$ over $\Z/p^2\Z$. Therefore, $\on{Tot}\check{C}(\tilde{U},\Omega_{\tilde{X}/W_2(k)})$ is flat over $\Z/p^2\Z$, and so $(\on{Tot}\check{C}(U,\Omega_{X/k}),\check{\theta})$ is reduced. As recalled at the end of \Cref{def-steenrod}, it now follows from \cite[Proposition 2.3(v)]{may1970general} that the Bockstein homomorphism $\beta$ of May's setting is defined on $H^*(\on{Tot}\check{C}({U},\Omega_{{X}/k}))$. Using (\ref{iso-to-reduced}), we obtain a Bockstein on $\mathbb{H}^*(\tilde{X},\Omega_{\tilde{X}/k})$. The conclusion now follows from \Cref{bockstein-properties}.
\end{proof}
	
	\begin{rmk}\label{bockstein-totaro}
	In the setting of \Cref{bockstein}, let $\iota:X\hookrightarrow \tilde{X}$ be the natural closed embedding. Following an idea of Totaro \cite[Proof of Theorem 11.1]{totaro2018hodge}, we may define a Bockstein homomorphism
	$H^*_{\on{dR}}(X/\F_p)\to H^{*+1}_{\on{dR}}(X/k)$ as the connecting homomorphism in the long exact sequence associated to \[0\to \iota_*\Omega_{X/k}\to \Omega_{\tilde{X}/W_2(k)}\to \iota_*\Omega_{X/k}\to 0.\]
	If $X$ is a scheme, one can give a definition even when without assuming that $X$ admits a lifting to $W_2(k)$, by using the crystalline Poincar\'e lemma and the short exact sequence
	\[0\to \iota_*\mc{O}_{X/k}\to \mc{O}_{X/W_2(k)}\to \iota_*\mc{O}_{X/k}\to 0.\]
	As we now show, \Cref{mainthm}(c) implies that Totaro's Bockstein is different from the Bockstein of \Cref{bockstein}. Recall from (\ref{dr-sing-orth}) that \[H^*_{\on{dR}}(B\on{O}_{2r}/\F_2)=\F_2[u_1,\dots,u_{2r}],\] where $|u_i|=i$. Then Totaro's homomorphism is non-trivial on the $u_{2a}$ (see the proof of \cite[Theorem 11.1]{totaro2018hodge}). On the other hand, by \Cref{mainthm}(c) all Steenrod operations, and in particular $\beta=\on{Sq}^1$, are trivial on the $u_{2a}$. 
	\end{rmk}

	\section{Proofs of Theorems \ref{main-steenrod} and \ref{mainthm}}\label{proof}

	\begin{proof}[Proof of \Cref{main-steenrod}]
	The naturality follows from \Cref{functorial}. Properties (i), (ii), (iii) hold by \Cref{sheaf-properties}. Property (iv) follows from \Cref{approximate}, \Cref{cris-compatible} and \Cref{steenrod-0}(a). Property (v) was proved in \Cref{bockstein}.	
	\end{proof}

	\begin{proof}[Proof of \Cref{mainthm}(a)]
		We denote by $B_{\on{simp}}G$ the simplicial classifying space of $G$; see \cite[Chapter 16, \S 5]{may1999concise}. Since $G$ is a finite discrete group, $B_{\on{simp}}G$ is a simplicial set. It is the quotient of the contractible simplicial set $E_{\on{simp}}G$ by the free action of $G$. We write $B_{\on{top}}G$ and $E_{\on{top}}G$ for the geometric realizations of $B_{\on{simp}}G$ and $E_{\on{simp}}G$, respectively; they are CW complexes.
		
		Let $\F_p\to I$ be an injective resolution of sheaves of $\F_p$-vector spaces on $B_{\on{top}}G$. Let $\theta:W\otimes \Gamma(B_{\on{top}}G,I)^{\otimes p}\to \Gamma(B_{\on{top}}G,I)$ be a $\pi$-equivariant homomorphism of $\F_p[\pi]$-complexes such that $(\Gamma(B_{\on{top}}G,I), \theta)$ is an element of $\mc{C}(p)$, as constructed in \Cref{definesteenrod}. By \Cref{epstein=may} and \cite[Theorem p. 206]{epstein1966steenrod}, the induced Steenrod operations on the sheaf cohomology ring $H^*(B_{\on{top}}G,\F_p)$ are the classical topological Steenrod operations.
		
		Let $X$ be a paracompact and Hausdorff topological space (for example, a CW complex). Let $\on{Op}(X)$ denote the site of open embeddings of $X$, and let $\on{Homeo}(X)$ be the site whose objects are local homeomorphisms $Y\to X$, whose morphisms are continuous maps $Y\to Y'$ over $X$ (they are automatically local homeomorphisms), and whose covers are families $\set{U_i\to Y}$ of jointly surjective local homeomorphisms. Since open embeddings are local homeomorphisms, there is an obvious morphism of sites $\on{Op}(X)\to \on{Homeo}(X)$ which by Verdier's Comparison Theorem \cite[III, Th\'eor\`eme 4.1]{sga4I} induces an equivalence of topoi; see also the beginning of \cite[XI, \S 4]{sga4III}. By \Cref{functorial}, we may compute the Steenrod operations on $H^*(B_{\on{top}}G,\F_p)$ while viewing $\F_p$ as a sheaf in either $\on{Op}(B_{\on{top}}G)$ or $\on{Homeo}(B_{\on{top}}G)$. It follows that we may apply \Cref{cech-to-derived} to the topos associated to $\on{Homeo}(B_{\on{top}}G)$ and to the cover $E_{\on{top}}G\to B_{\on{top}}G$. Let $U$ denote the \v{C}ech nerve of $E_{\on{top}}G\to B_{\on{top}}G$. We have a morphism of simplicial topological spaces $U\to B_{\on{simp}}G$, given by the unique map $E_{\on{top}}G\to \set{*}$ in level $0$, and by projections onto the $G^{n-1}$ factor in level $n$, for all $n\geq 1$. 
		
		Recall that the big \'etale topos of $BG$ is the topos associated to the site whose objects are schemes over $BG$, whose arrows are morphisms of schemes over $BG$, and whose covers are families of jointly surjective \'etale morphisms of schemes over $BG$. By Verdier's Comparison Theorem, it is also the topos associated to the site whose objects are representable morphisms of algebraic stacks $X\to BG$, whose arrows are morphisms of algebraic stacks over $BG$, and whose covers are families of jointly surjective \'etale morphisms. Since $BG$ is a terminal object in the second site, we may apply \Cref{cech-to-derived} to the big \'etale topos of the algebraic stack $BG$, and to the universal $G$-torsor $\Spec \F_p\to BG$. The \v{C}ech nerve of this morphism is exactly $B_{\on{simp}}G$. 
		
		By \Cref{remark-functorial}, the projection $U\to B_{\on{simp}}G$ and the inclusion $\F_p\hookrightarrow \Omega_{BG/\F_p}$ induce the commutative squares of $\pi$-equivariant maps:
		\[
		\begin{tikzcd}
		\mc{Z}(p)\otimes \on{Tot}\check{C}(U,\F_p)^{\otimes p} \arrow[r] \arrow[d,"\wr"] & \on{Tot}\check{C}(U,\F_p)\arrow[d,"\wr"] \\
		\mc{Z}(p)\otimes \on{Tot}\check{C}(B_{\on{simp}}G,\F_p)^{\otimes p} \arrow[r] & \on{Tot}\check{C}(B_{\on{simp}}G,\F_p)
		\end{tikzcd}
		\]	
		and
		\[
		\begin{tikzcd}
		\mc{Z}(p)\otimes \on{Tot}\check{C}(B_{\on{simp}}G,\F_p)^{\otimes p} \arrow[r] \arrow[d,"\wr"] & \on{Tot}\check{C}(B_{\on{simp}}G,\F_p)\arrow[d,"\wr"] \\
		\mc{Z}(p)\otimes \on{Tot}\check{C}(B_{\on{simp}}G,\Omega_{BG/\F_p})^{\otimes p} \arrow[r] & \on{Tot}\check{C}(B_{\on{simp}}G,\Omega_{BG/\F_p}). 
		\end{tikzcd}
		\]	
		
		The  vertical maps in the first square are isomorphisms because $E_{\on{top}}G$ is connected and $\F_p$ is a constant sheaf. Since $G$ is discrete, we have $\eta_j^{-1}\Omega^q_{BG/\F_p}=\Omega^q_{G^j/\F_p}=0$ for all $q>0$ and $j\geq 0$, and $\eta_j^{-1}\mc{O}_{BG/\F_p}=\mc{O}_{G^j/\F_p}$; see \cite[Tag 06TU]{stacks-project}. (Recall that the $\eta_j^{-1}$ are pullbacks of sheaves on big \'etale sites.) Thus the vertical maps in the second square are isomorphisms too. 
		Combining the two squares and pre-composing with the quasi-isomorphism $W\to\mc{Z}(p)$ of (\ref{w-e-z}), we obtain a square of $\pi$-equivariant maps
		\[
		\begin{tikzcd}
		W\otimes \on{Tot}\check{C}(U,\F_p)^{\otimes p} \arrow[r,"\check{\theta}_t"] \arrow[d] & \on{Tot}\check{C}(U,\F_p)\arrow[d] \\
		W\otimes \on{Tot}\check{C}(B_{\on{simp}}G,\Omega_{BG/\F_p})^{\otimes p} \arrow[r,"\check{\theta}_a"] & \on{Tot}\check{C}(B_{\on{simp}}G,\Omega_{BG/\F_p}), 
		\end{tikzcd}
		\]
		where $\check{\theta}_t$ and $\check{\theta}_a$ are given by (\ref{check-global}). We conclude that  the second projection $U\to B_{\on{simp}}G$ induces an isomorphism \[\on{Tot}(\check{C}(U,\F_p),\check{\theta}_t)\cong (\on{Tot}\check{C}(B_{\on{simp}}G,\Omega_{BG/\F_p}),\check{\theta}_a)\] in $\mc{C}(p)$, and this completes the proof.
 	\end{proof}

	In order to prove \Cref{mainthm}(b) and (c), we need some auxiliary computations.
	
	\begin{lemma}
		Let $H^*_{\on{dR}}(\P^n/\F_p)=\F_p[x]/(x^{n+1})$, where $x$ has degree $2$.
		\begin{enumerate}[label=(\alph*)]
			\item If $p=2$, then $\on{Sq}(x)=x^2$.
			\item If $p>2$, then $\on{P}(x)=x^p$ and $\beta(x)=0$.
		\end{enumerate}
	\end{lemma}

	\begin{proof}
	By \Cref{main-steenrod}(iv), the negative Steenrod operations on $\P^n$ are trivial. Recall that  $H^i(\P^n,\mc{O}_{\P^n})$ is zero for $i>0$. By \Cref{main-steenrod}(iv), we deduce that $\on{P}^0$ and $\on{Sq}^0$ are zero on $H^*_{\on{dR}}(\P^n/\F_p)$. Since $\beta(x)$ has degree $3$, necessarily $\beta(x)=0$.
	\end{proof}
	
		\begin{lemma}\label{gm}
		Let $H^*_{\on{dR}}(B\G_{\on{m}}/\F_p)=\F_p[x]$, where $x$ has degree $2$.
		\begin{enumerate}[label=(\alph*)]
			\item If $p=2$, then $\on{Sq}(x)=x^2$.
			\item If $p>2$, then $\on{P}(x)=x^p$ and $\beta(x)=0$.
		\end{enumerate}
	\end{lemma}
	
	\begin{proof}
		By \Cref{main-steenrod}(iv), all negative Steenrod operations on $H^*_{\on{dR}}(B\G_{\on{m}}/\F_p)$ are trivial. Let $f:\P^1\to B\G_{\on{m}}$ be the morphism corresponding to the $\G_{\on{m}}$-torsor $\A^{2}\setminus\set{0}\to \P^1$. We may fix an isomorphism $H^*_{\on{dR}}(\P^1/\F_p)\cong \F_p[x]/(x^2)$ so that the induced ring homomorphism $f^*:H^*_{\on{dR}}(B\G_{\on{m}}/\F_p)\to H^*_{\on{dR}}(\P^1/\F_p)$ is given by reduction modulo $x^2$. By \Cref{functorial}, $f^*$ is compatible with Steenrod operations. 
		
		(a) Write $\on{Sq}(x)=ax+x^2$ for some $a\in \F_2$. We have	\[0=\on{Sq}^0(f^*(x))=f^*(\on{Sq}^0(x))=f^*(ax)=af^*(x),\] which implies that $a=0$, hence $\on{Sq}(x)=x^2$.
		
		(b) Since $\beta(x)$ has degree $3$ and $H^*_{\on{dR}}(B\G_{\on{m}}/\F_p)$ is concentrated in even degrees, we have $\beta(x)=0$. Write $\on{P}(x)=ax+x^p$. Then
		\[0=\on{P}^0(f^*x)=f^*(\on{P}^0(x))=f^*(ax)=af^*(x),\] hence $a=0$ and $\on{P}(x)=x^p$.
	\end{proof}

	Let $k$ be a field, let $G$ be a connected reductive $k$-group, and let $p$ be a prime number. Assume first that $k$ is algebraically closed. Let $T$ be a maximal torus of $G$, let $\hat{T}$ be the character group of $T$, let $B$ be a Borel subgroup of $G$ containing $T$, and let $n:=\dim G/B$. We have a natural group homomorphism $\hat{T}\to CH^1(G/B)$. We obtain a homomorphism \begin{equation}\label{torsion-def}\on{Sym}^n(\hat{T})\to CH^n(G/B)\xrightarrow{\deg}\Z\xrightarrow{\pi} \Z/p\Z,\end{equation} where $\on{deg}$ is the degree map and $\pi$ is the projection modulo $p$. We say that $p$ is a torsion prime for $G$ if the composition (\ref{torsion-def}) is zero. If $k$ is an arbitrary field, we say that $p$ is a torsion prime for $G$ if $p$ is a torsion prime for $G_{\cl{k}}$. This is the definition given in \cite[p. 1592]{totaro2018hodge}, to which we refer for other equivalent formulations.

	\begin{proof}[Proof of \Cref{mainthm}(b)]
		Let $T$ be a maximal torus of $G$, let $\mathfrak{g}$ and $\mathfrak{t}$ be the Lie algebras of $G$ and $T$, $\mathfrak{g}^*$ and $\mathfrak{t}^*$ be their duals, and $\on{Sym}(\mathfrak{g}^*)$ and $\on{Sym}(\mathfrak{t}^*)$ be the symmetric $\F_p$-algebras on $\mathfrak{g}$ and $\mathfrak{t}$. Fix an integer $i\geq 0$. We have a commutative diagram
		\begin{equation}\label{torus-pullback}
		\begin{tikzcd}
		H^0(G,\on{Sym}^i(\mathfrak{g}^*))\arrow[d]  & H^i(BG,\Omega_{BG/\F_p}^i)\arrow[r] \arrow[l,swap,"\sim"] \arrow[d] & H^{2i}_{\on{dR}}(BG/\F_p) \arrow[d]  \\
		H^0(T,\on{Sym}^i(\mathfrak{t}^*)) & H^i(BT,\Omega_{BT/\F_p}^i) \arrow[l,swap,"\sim"] \arrow[r] &  H^{2i}_{\on{dR}}(BT/\F_p),
		\end{tikzcd}
		\end{equation}
		where the vertical maps are pullbacks, the horizontal maps on the left are the isomorphisms of \cite[Corollary 2.2]{totaro2018hodge}, and the horizontal maps on the right arise in the Hodge-de Rham spectral sequence for $BG$ and $BT$; see \cite[Lemma 8.2]{totaro2018hodge}.	
		
		Let $T$ be a split maximal torus of $G$, and let $B$ be a Borel subgroup of $G$ containing $T$. We have a spectral sequence of $k$-algebras
		\[E_2^{ij}:=H^i_{\on{H}}(BG/\F_p)\otimes H^j_{\on{H}}((G/B)/\F_p)\Rightarrow H^{i+j}_{\on{H}}(BT/\F_p),\]
		where $H^i_{\on{H}}(-/\F_p):= \oplus_l H^l(-,\Omega^{i-l})$; see \cite[Proposition 9.3]{totaro2018hodge}. Since $p$ is not a torsion prime for $G$, the spectral sequence degenerates; see the last paragraph of the proof of \cite[Theorem 9.1]{totaro2018hodge}. In particular, the pullback map $H^i(BG,\Omega_{BG/\F_p}^i)\to H^i(BT,\Omega_{BT/\F_p}^i)$ appearing in (\ref{torus-pullback}) is injective.
		
		It follows from \cite[Corollary 2.2, Theorem 4.1]{totaro2018hodge} that $H^i(BT,\Omega_{BT/\F_p}^j)=0$ when $i\neq j$. Thus $H^i(BG,\Omega_{BG/\F_p}^j)=0$ when $i\neq j$. Thus in the Hodge-de Rham spectral sequences
		\[E_1^{ij}:=H^j(X,\Omega_{X/\F_p}^i)\Rightarrow H^{i+j}_{\on{dR}}(X/\F_p)\] 
		 for $X=BG,BT$, the $E_1$-pages are concentrated in the diagonal $i=j$. This implies that the spectral sequences degenerate. From this we obtain that $H^j_{\on{dR}}(BG/\F_p)=0$ for odd $j$, and that the right horizontal maps in (\ref{torus-pullback}) are isomorphisms. 
		 
		Summarizing, we have showed that all horizontal arrows in (\ref{torus-pullback}) are isomorphisms, and all vertical arrows are injective. Since $H^j_{\on{dR}}(BG/\F_p)=0$ for odd $j$, this shows that the pullback $H^*_{\on{dR}}(BG/\F_p)\to H^*_{\on{dR}}(BT/\F_p)$ is injective.
	\end{proof}

\begin{rmk}
	One could also prove \Cref{mainthm}(b) by establishing the injectivity of $H^0(G,\on{Sym}^i(\mathfrak{g}^*))\to H^0(T,\on{Sym}^i(\mathfrak{t}^*))$, as follows. It is enough to treat the case when $G$ is semisimple. Then the injectivity follows from \cite[Theorem 8.1]{totaro2018hodge} (due to T. A. Springer and R. Steinberg \cite[Section II.3.17']{springer1970conjugacy}, and P.-E. Chaput and M. Romagny \cite[Theorem 1.1]{chaput2010adjoint}) when $p>2$, or when $p=2$ and $\mathfrak{g}$ does not have factors of the form $\mathfrak{Sp}_{2n}$. In the remaining cases, it is a consequence of results of Chaput-Romagny; see the last paragraph of the proof of \cite[Theorem 9.2]{totaro2018hodge}.
\end{rmk}

\begin{lemma}\label{mup}
	Let $H^*_{\on{dR}}(B\mu_p/\F_p)=\F_p[t,v]/(v^2)$, where $|t|=2$ and $|v|=1$.
	\begin{enumerate}[label=(\alph*)]
		\item If $p=2$, then $\on{Sq}(t)=t^2$ and $\on{Sq}(v)=0$.
		\item If $p>2$, then $\on{P}(t)=t^p$, $\on{P}(v)=0$, $\beta(t)=0$, $\beta(v)=0$.
	\end{enumerate}
\end{lemma}

\begin{proof}
	From the proof of \cite[Proposition 10.1]{totaro2018hodge}, $t$ is defined as the pullback of $c_1\in H^2_{\on{dR}}(B\G_{\on{m}}/\F_p)$ induced by the inclusion $\mu_p\hookrightarrow\G_{\on{m}}$. By \Cref{gm}, we deduce that $\on{Sq}(t)=t^2$ if $p=2$, and that $\on{P}(t)=t^p$ and $\beta(t)=0$ if $p>2$. 
	
	Since $|v|=1$ and $v^2=0$, we have $\on{Sq}(v)=\on{Sq}^0(v)$ and $\on{P}(v)=\on{P}^0(v)$. From the proof of \cite[Proposition 10.1]{totaro2018hodge}, we see that $v$ belongs to the image of the differential $H^0(B\mu_p,\Omega^1_{B\mu_p/\F_p})\to H^1_{\on{dR}}(B\mu_p/\F_p)$ arising from the Hodge-de Rham spectral sequence for $B\mu_p$. By \Cref{steenrod-0}(b), we deduce that $\on{Sq}^0(v)=0$ and $\on{P}^0(v)=0$.
	
	To prove that $\beta(v)=0$, one may apply the \v{C}ech-theoretic interpretation given in of \Cref{bock-cech} to the presentation of $B\mu_p$ as $[\G_{\on{m}}/\G_{\on{m}}]$, where $\G_{\on{m}}$ acts on itself via the $p$-th power map. The \v{C}ech cover associated to the $\G_{\on{m}}$-torsor $\G_{\on{m}}\to [\G_{\on{m}}/\G_{\on{m}}]$ consists of $\G_{\on{m}}^i$ in degree $i$, for all $i\geq 0$. As we have no use for $\beta(v)=0$ in the paper, we leave this computation to the reader.
\end{proof}

	The last ingredient is the following special case of the K\"unneth formula in de Rham cohomology.

	\begin{lemma}\label{kunneth}
		Let $k$ be a field, let $G$ and $H$ be linear algebraic $k$-groups. Assume that the Hodge-de Rham spectral sequences for $BG$ and $BH$ degenerate.
		
		(a) The Hodge-de Rham spectral sequence for $B(G\times_k H)$ degenerates.
		 
		(b) Assume further that $H^i(BG,\Omega^j_{BG/k})=H^i(BH,\Omega^j_{BH/k})=0$ for all $i\neq j,j+1$. Then the projections $G\times_k H\to G$ and $G\times_k H\to H$ induce an isomorphism of $k$-algebras
		\[H^*_{\on{dR}}(BG/k)\otimes H^*_{\on{dR}}(BH/k)\xrightarrow{\sim} H^*_{\on{dR}}(B(G\times_kH)/k).\]
	\end{lemma}

	\begin{proof}
		By the K\"unneth formula in Hodge cohomology \cite[Proposition 5.1]{totaro2018hodge}, the projections of $G\times_kH$ onto $G$ and $H$ induce an isomorphism
		\begin{equation}\label{kunneth-hodge}H^*_{\on{H}}(BG/k)\otimes H^*_{\on{H}}(BH/k)\xrightarrow{\sim} H^*_{\on{H}}(B(G\times_kH)/k),\end{equation} where $H^i_{\on{H}}(-/\F_p):= \oplus_l H^l(-,\Omega^{i-l})$ is Hodge cohomology. This isomorphism is multiplicative and respects the bigrading in Hodge cohomology.
		
		(a) The Hodge-de Rham spectral sequence is a spectral sequence of $k$-algebras, and its $E_1$-page is Hodge cohomology. By assumption, all differentials are trivial on elements coming from $H^*_{\on{H}}(BG/k)$ and $H^*_{\on{H}}(BH/k)$, hence on all elements of $H^*_{\on{H}}(B(G\times_kH)/k)$.
		
		(b) It follows that $H^i(B(G\times_kH),\Omega^j_{B(G\times_kH)/k})=0$ for $i\neq j,j+1$. Thus, the $E_1$-page of the Hodge-de Rham spectral sequences for $G,H,G\times_kH$ are concentrated in bidegrees $(i,i),(i+1,i)$. By the assumptions and (a), the Hodge-de Rham spectral sequences for $G,H,G\times_k$ are degenerate. This yields a commutative square
		\[
		\begin{tikzcd}
		H^*_{\on{H}}(BG/k)\otimes H^*_{\on{H}}(BH/k) \arrow[r,"\sim"]\arrow[d,"\wr"] & H^*_{\on{H}}(B(G\times_kH)/k) \arrow[d,"\wr"]  \\
		H^*_{\on{dR}}(BG/k)\otimes H^*_{\on{dR}}(BH/k) \arrow[r] &  H^*_{\on{dR}}(B(G\times_kH)/k),
		\end{tikzcd}
		\]
		from which we obtain the desired isomorphism.		 
	\end{proof}

	\begin{proof}[Proof of \Cref{mainthm}(c)]
	We start with $B\on{O}_2$. We have
	\[H^*_{\on{dR}}(B\on{O}_2/\F_2)=\F_2[u_1,u_2],\qquad |u_1|=1, |u_2|=2.\] We think of $\on{O}_2$ as the isometry group of the quadratic form $q$ on $\A^2_{\F_2}$ given by $q(x,y)=xy$. There is a subgroup $H$ of $\on{O}_2$, isomorphic to $\Z/2\Z\times \mu_2$, where $\Z/2\Z$ acts by switching $x$ and $y$, and $\mu_2$ acts by scaling on $\A^2_{\F_2}$. By \Cref{mainthm}(a) we have \[H^*_{\on{dR}}(B(\Z/2\Z)/\F_2)=\F_2[s],\qquad |s|=1,\qquad \on{Sq}(s)=s+s^2,\] and by \Cref{mup}(a) we have \[H^*(B\mu_2/\F_2)=\F_2[t,v]/(v^2),\qquad |t|=2, |v|=1,\qquad \on{Sq}(t)=t^2,\quad \on{Sq}(v)=0.\] By \Cref{kunneth},  we have \[H^*_{\on{dR}}(BH/\F_2)=\F_2[s,t,v]/(v^2),\] compatibly with the projection maps $BH\to B(\Z/2\Z)$ and $BH\to B\mu_2$. Applying \Cref{functorial} to the projections $BH\to B\mu_2$ and $BH\to B(\Z/2\Z)$, we obtain
	\[\on{Sq}(s)=s+s^2,\quad \on{Sq}(t)=t^2,\quad \on{Sq}(v)=0.\] The induced homomorphism $H^*(B\on{O}_2/\F_2)\to H^*(BH/\F_2)$ sends $u_1\mapsto s$ and $u_2\mapsto t+sv$, and in particular it is injective; see \cite[p. 1602]{totaro2018hodge}. Note that $\on{Sq}(t+sv)=t^2$ and that $u_2^2\mapsto t^2$. It follows that 
	\[\on{Sq}(u_1)=u_1+u_1^2, \quad \on{Sq}(u_2)=u_2^2.\]
	For every $r\geq 1$, by \Cref{kunneth} we may write 
	\[H^*_{\on{dR}}(B\on{O}_2^r/\F_2)=\F_2[s_1,\dots,s_r,t_1,\dots,t_r].\]	
	Here $s_i$ and $t_i$ are the pullback of the classes $u_1$ and $u_2$ in $H^*(B\on{O}_2/\F_2)$ along the $i$-th projection $B\on{O}_2^r\to B\on{O}_2$, respectively. Applying \Cref{functorial} to the projections, we see that	
	\[\on{Sq}(s_i)=s_i+s_i^2,\qquad \on{Sq}(t_i)=t_i^2.\]	
	By \cite[Theorem 11.1, Lemma 11.3]{totaro2018hodge}, we have	
	\[H^*_{\on{dR}}(B\on{O}_{2r}/\F_2)=\F_2[u_1,u_2,\dots,u_{2r}],\] where $|u_i|=i$ for every $i=1,\dots,2r$. Moreover, there exists a closed subgroup embedding $\iota: \on{O}_2^r\hookrightarrow \on{O}_{2r}$ such that the pullback map \[\iota^*: H^*_{\on{dR}}(B\on{O}_{2r}/\F_2)\to H^*_{\on{dR}}(B\on{O}_2^r/\F_2)\] is injective, and is given by
	\[\iota^*(u_{2a})=\sum_{1\leq i_1<\dots <i_a\leq r}t_{i_1}\cdots t_{i_a}, \quad \iota^*(u_{2a+1})=\sum_{m=1}^rs_m\sum_{1\leq i_1<\dots <i_a\leq r, i_h\neq m}t_{i_1}\cdots t_{i_a}.\]
	Since $\on{Sq}(t_i)=t_i^2$ and $\iota^*$ is injective, we have $\on{Sq}(u_{2a})=u_{2a}^2$. One proves by ascending induction on $0\leq d\leq a$ that
		\begin{equation}\label{orth3}
	\iota^*(u_{4a+1}+\sum_{t=2a-2d}^{2a-1}u_{2a-t}u_{2a+1+t})=\sum_{m=1}^r s_m\sum_{j_h\neq m} t_{j_1}\dots t_{j_d}\sum_{i_h\neq m} t_{i_1}\dots t_{i_{2a-d}}. 
	\end{equation}
	\begin{equation}\label{orth4}
	\iota^*(u_{4a+1}+\sum_{t=2a-2d-1}^{2a-1}u_{2a-t}u_{2a+1+t})=\sum_{m=1}^r s_m\sum_{\exists h: j_h=m} t_{j_1}\dots t_{j_d}\sum_{i_h\neq m} t_{i_1}\dots t_{i_{2a-d}},
	\end{equation}
	In the sums, $h$ is arbitrary, $1\leq j_1<\dots <j_d\leq r$ and $1\leq i_1<\dots < i_{2a-d}\leq r$. We define $t_{j_h},s_{i_h}:=0$ when $h>a$. When $d=0$, (\ref{orth3}) coincides with the previous expression of $\iota^*(u_{4a+1})$ in terms of the $s_i$ and $t_i$, and so is true. One then uses (\ref{orth3}) for $d$ to prove (\ref{orth4}) for $d$, and (\ref{orth4}) for $d$ to prove (\ref{orth3}) for $d+1$. By induction, this proves the formulas for all $0\leq d\leq a$. When $d=a$, (\ref{orth3}) becomes
	\[\iota^*(u_{4a+1}+\sum_{t=0}^{2a-1}u_{2a-t}u_{2a+1+t})=\sum_{m=1}^r s_m\sum_{j_h\neq m} t_{i_1}^2\dots t_{i_a}^2=\iota^*(\on{Sq}(u_{2a+1}))+\iota^*(u_{2a+1})^2.\] Since $\iota^*$ is injective, this implies \Cref{mainthm}(c) for $B\on{O}_{2r}$.
	
	The homomorphism \[H^*_{\on{dR}}(B\on{O}_{2r}/\F_2)\to H^*_{\on{dR}}(B\on{SO}_{2r}/\F_2)\] induced by the canonical inclusion $\on{SO}_{2r}\hookrightarrow \on{O}_{2r}$ sends $u_1\mapsto 0$ and $u_i\mapsto u_i$ for $i\geq 2$. This implies \Cref{mainthm}(c) for $\on{SO}_{2r}$. 
	
	Consider now the composition of closed embeddings \[\on{SO}_{2r+1}\hookrightarrow \on{O}_{2r+1}\hookrightarrow \on{O}_{2r+2},\] where $\on{O}_{2r+1}$ is embedded as the stabilizer subgroup of a non-zero vector of the standard $(2r+2)$-dimensional representation of $\on{O}_{2r+2}$. The induced homomorphism $H^*_{\on{dR}}(\on{O}_{2r+2}/\F_2)\to H^*_{\on{dR}}(\on{SO}_{2r+1}/\F_2)$ sends $u_1,u_{2r+2}\mapsto 0$ and $u_i\mapsto u_i$ for $2\leq i\leq 2r+1$, proving \Cref{mainthm}(c) for $\on{SO}_{2r+1}$. 
	
	Finally, recall that $\on{O}_{2r+1}\cong \on{SO}_{2r+1}\times \mu_2$. The K\"unneth isomorphism
	\[H^*_{\on{dR}}(B\on{SO}_{2r+1}/\F_2)\otimes H^*_{\on{dR}}(B\mu_2/\F_2)\xrightarrow{\sim} H^*_{\on{dR}}(B\on{O}_{2r+1}/\F_2)\] sends $u_i\mapsto u_i$ for all $i$, $v\mapsto v_1$ and $c\mapsto c_1$ (we are using the notation of \Cref{mup} for the cohomology of $B\mu_2$). Taking \Cref{mup} into account, this proves \Cref{mainthm}(c) for $B\on{O}_{2r+1}$.
	\end{proof}

	\section*{Acknowledgments}
	I would like to thank Ben Williams for a number of useful discussions on this topic, Benjamin Antieau and Samuel Bach for helpful conversations, and my advisor Zinovy Reichstein for giving me direction at various stages of this work.

\end{document}